\documentclass[reqno]{amsart}
\usepackage[margin=1.15in]{geometry}

\usepackage{hyperref}
\hypersetup{pdftitle={Boundary estimates for singular elliptic problems involving a gradient term},pdfauthor={Phuong Le}}
\usepackage{cite}
\usepackage{enumitem}

\theoremstyle{plain}
\newtheorem{theorem}{Theorem}

\newtheorem{corollary}[theorem]{Corollary}
\newtheorem{lemma}[theorem]{Lemma}

\theoremstyle{definition}

\theoremstyle{remark}
\newtheorem{remark}{Remark}

\DeclareMathOperator*{\essinf}{ess\,inf}

\allowdisplaybreaks

\begin{document}

	\author{Phuong Le}
	\address{Phuong Le$^{1,2}$ (ORCID: 0000-0003-4724-7118)\newline
		$^1$Faculty of Economic Mathematics, University of Economics and Law, Ho Chi Minh City, Vietnam; \newline
		$^2$Vietnam National University, Ho Chi Minh City, Vietnam}
	\email{phuongl@uel.edu.vn}


	\title[Boundary estimates for singular elliptic problems]{Boundary estimates for singular elliptic problems involving a gradient term}
	\begin{abstract}
We study bounded positive weak solutions of the singular quasilinear problem
$-\Delta_p u+\vartheta|\nabla u|^q=u^{-\gamma}+f(u)$ in a bounded $C^2$ domain, with zero Dirichlet boundary data, where $p>1$, $\gamma>0$, $0<q\le p$, $\vartheta\ge0$, and $f$ is continuous up to zero and locally Lipschitz on $(0,\infty)$. We establish sharp boundary asymptotics for the solution and its directional derivatives. For $\gamma>1$ the leading profile is a universal power of the boundary distance; for $\gamma=1$ it is logarithmic, with the sharp constant $(p/(p-1))^{1/p}$. In both cases the normalized gradient converges uniformly to the inward normal multiplied by the corresponding constant. For $0<\gamma<1$ we prove global $C^{1,\alpha}$ regularity and a classical Hopf boundary estimate. We also obtain the optimal Sobolev integrability threshold and the sharp logarithmic modulus of continuity. Finally, positivity of the total source on the solution range yields symmetry and monotonicity in symmetric domains strictly convex in one direction; when $N\ge2$ and $\vartheta>0$, we assume $q\ge\max\{1,p/2\}$.
\end{abstract}

	\maketitle
	\enlargethispage{4pt}

	\section{Introduction}
	We study the behavior near the boundary of solutions to the quasilinear elliptic problem
	\begin{equation}\label{main}
		\begin{cases}
			-\Delta_p u + \vartheta |\nabla u|^q = \dfrac{1}{u^\gamma} + f(u) &\text{ in } \Omega,\\
			u>0 &\text{ in } \Omega,\\
			u=0 &\text{ on } \partial\Omega,
		\end{cases}
	\end{equation}
	where we always assume that $\Omega$ is a bounded $C^2$ domain of $\mathbb{R}^N$, $N\ge1$, $p>1$, $\gamma>0$, $0<q\le p$, $\vartheta\ge0$, and $f:[0,+\infty)\to\mathbb{R}$ is continuous and locally Lipschitz on $(0,+\infty)$. No Lipschitz condition at zero is needed: the boundary arguments use only boundedness of $f$ on bounded intervals, while the interior comparison arguments use its Lipschitz bound on compact subintervals of $(0,+\infty)$.

	We call $u \in W^{1,p}_{\rm loc}(\Omega) \cap L^\infty(\Omega)$ a weak solution to the first equation of \eqref{main} if
	\[
	\int_\Omega |\nabla u|^{p-2} \langle \nabla u, \nabla \varphi \rangle + \vartheta \int_\Omega |\nabla u|^q \varphi = \int_\Omega \frac{\varphi}{u^\gamma} + \int_\Omega f(u) \varphi \quad\text{ for every } \varphi \in C^1_c(\Omega).
	\]
	Moreover, the positivity of $u$ is understood as
	\[
	\essinf_{x\in K} u(x) > 0 \quad \text{ for every } K \subset\subset \Omega,
	\]
	and the zero Dirichlet boundary condition also has to be understood in the weak meaning
	\[
	(u - \varepsilon)^+ \in W^{1,p}_0(\Omega) \quad \text{ for every } \varepsilon>0.
	\]
	In the same spirit, for $v_1,v_2 \in W^{1,p}_{\rm loc}(\Omega) \cap L^\infty(\Omega)$, we say that $v_1 \le v_2$ on $\partial\Omega$ if
	\[
	(v_1 - v_2 - \varepsilon)^+ \in W^{1,p}_0(\Omega) \quad \text{ for every } \varepsilon>0.
	\]

	We show below that $u\in C^{1,\alpha}_{\rm loc}(\Omega)\cap C(\overline{\Omega})$ for some $\alpha\in(0,1)$. For $\gamma\ge1$, the singularity prevents $C^1$ regularity up to the boundary and the gradient blows up there; for sufficiently large $\gamma$, even membership in $W^{1,p}_0(\Omega)$ fails (see \cite{MR427826,MR1037213} for $p=2$, and Corollary \ref{co:W1p0} below for general $p$). Furthermore, the difficulty in understanding problem \eqref{main} lies in the fact that both sides of the problem may exhibit singularity near $\partial\Omega$. A simpler situation was investigated by Esposito and Sciunzi in \cite{MR4044739} for the singular quasilinear elliptic problem without the gradient term
	\begin{equation}\label{singular}
		\begin{cases}
			-\Delta_p u = \dfrac{1}{u^\gamma} + f(u) &\text{ in } \Omega,\\
			u>0 &\text{ in } \Omega,\\
			u=0 &\text{ on } \partial\Omega.
		\end{cases}
	\end{equation}
	This is exactly problem \eqref{main} with $\vartheta=0$.

	Various authors have studied singular elliptic problems without a gradient term in the past. We refer in particular to the pioneering papers by Crandall, Rabinowitz and Tartar \cite{MR427826} and by Lazer and McKenna \cite{MR1037213}, where they combine the method of subsolutions
	and supersolutions with the approximation technique to prove the existence, uniqueness and qualitative properties of solutions to
	\begin{equation}\label{pure_singular}
		\begin{cases}
			-\Delta_p u = \dfrac{a(x)}{u^\gamma} &\text{ in } \Omega,\\
			u>0 &\text{ in } \Omega,\\
			u=0 &\text{ on } \partial\Omega
		\end{cases}
	\end{equation}
	with $p=2$ (see \cite{MR1263903} for further insights). We also refer to the work by Boccardo and Orsina \cite{MR2592976} on the same problem in a more general setting, exploiting truncation arguments and Schauder's fixed-point theorem. Later, the existence, uniqueness and regularity of solutions to problem \eqref{pure_singular} with $p>1$ were derived in \cite{MR3136107,MR3478284}. Another possible approach to problem \eqref{singular} with $p=2$ and a power reaction $f(u)=\lambda u^m$ is the variational one as in \cite{MR1022988,MR2159469,MR2099611} (see also \cite{MR2927112,MR3130522}).

	Gradient regularity for nonhomogeneous equations of $p$-Laplace type has also been studied through nonlinear potential theory. Duzaar and Mingione \cite{DM2010Potentials} established pointwise gradient estimates in terms of linear and nonlinear potentials of the datum, while \cite{DM2010Continuity} gives potential criteria for gradient continuity in the degenerate setting. For the singular problem \eqref{main}, the gradient may blow up at the boundary, and our aim is to determine the precise asymptotic behavior of its directional derivatives, including the logarithmic borderline regime. We refer to \cite{MR4658655,MR4333974} for symmetry results for solutions of \eqref{main} in the case $p=2$ and $q\in\{1,2\}$. Earlier existence results for the natural-growth case $q=p$, $\vartheta=1$, $f\equiv0$ and $1<p<N$ were obtained in \cite{MR4587608}, in the more general setting of singular problems with $L^1$ data.

The results below are a priori statements for every bounded solution; existence
is not asserted for arbitrary $f$ and arbitrary parameters. A broad nonempty
class is supplied by the recent preprint of Huynh and Le
\cite[Theorem 1.2]{HuynhLe2026}: when $q=p$, $\vartheta=1$ and
$f(s)=\lambda s^m$ with $0<m\le p$, a bounded positive solution exists for
every $\lambda>0$, every $\gamma>0$ and every $N\ge1$.
Here $m$ denotes the reaction exponent, whereas $q$ remains the gradient
exponent of \eqref{main}. Their existence proof is independent of the
boundary estimates established here. The solutions are continuous up to the
boundary, vanish there, and are $C^{1,\alpha}_{\rm loc}$ in the interior.
Consequently $(u-\epsilon)^+$ has compact support in $\Omega$ and belongs to
$W^{1,p}_0(\Omega)$ for every $\epsilon>0$, so these solutions lie in precisely
the weak class used in this paper.

The same existence conclusion holds for any $\vartheta>0$ in this
natural-growth case. Indeed, set
\[
b=\vartheta^{(\gamma+p-1)/p},\qquad
\widetilde\lambda=\lambda\vartheta^{-(\gamma+m)}.
\]
If $v$ is a solution supplied by that theorem on $b\Omega$, with reaction
$\widetilde\lambda v^m$, then $u(x)=\vartheta^{-1}v(bx)$ solves
\eqref{main} on $\Omega$ with $q=p$ and $f(u)=\lambda u^m$.
For $m>p$, \cite[Theorem 1.4]{HuynhLe2026} still gives existence for
sufficiently small $\lambda>0$. Thus the boundary conclusions apply to these
solutions, including those outside the energy space; in symmetric strictly
directionally convex domains, the symmetry theorem below also applies since
$s^{-\gamma}+\lambda s^m>0$.

	Now we recall the main result in \cite{MR4044739} (see also \cite{MR3912757} for $p=2$), which motivates our work. Throughout the paper, we denote the distance function for domain $\Omega$ by
	\[
	d_\Omega(x) := {\rm dist}(x,\partial\Omega) \quad\text{ for } x \in \Omega.
	\]
	For $\delta>0$, we denote by
	\[
	I_\delta(\Omega) := \{x\in\Omega \mid d_\Omega(x)<\delta\}
	\]
	the $\delta$--neighborhood of the boundary of $\Omega$. Since $\Omega$ is a $C^2$ domain, $ I_\delta(\Omega)$ has the unique nearest point property for small $\delta$ (see \cite{MR749908}). For such $\delta$ and $x\in I_\delta(\Omega)$, let $\hat x\in\partial\Omega$ be the point such that $|x-\hat x|=d_\Omega(x)$. Then we define
	\[
	\eta(x) = \frac{x-\hat x}{|x-\hat x|}.
	\]
	It was proved in \cite[Theorem 1.1]{MR4044739} that, in bounded $C^{2,\alpha}$ domains, for $\gamma>1$, every solution $u\in C^{1,\alpha}_{\rm loc}(\Omega)\cap C(\overline{\Omega})$ to \eqref{singular} and every $\beta\in(0,1)$, there exists a neighborhood $ I_\delta(\Omega)$ of $\partial\Omega$, depending on $u$ and $\beta$, such that $\frac{\partial u(x)}{\partial \nu}>0$
	whenever $\nu \in \mathbb{S}^{N-1}$ with $\langle\nu,\eta(x)\rangle\ge\beta$, where
	\[
	\mathbb{S}^{N-1} := \{x\in\mathbb{R}^N \mid |x|=1\}.
	\]
	This result is regarded as the Hopf boundary lemma for the problem. The Hopf boundary lemma is already known in the non-singular quasilinear setting (see \cite{MR768629}) and plays an important role in nonlinear analysis.

	Intuitively, since the gradient of $u$ may blow up near the boundary, for $\gamma\ge1$ we expect $\frac{\partial u(x)}{\partial \nu}$ to tend to $+\infty$ when $x$ approaches $\partial\Omega$ and $\langle\nu,\eta(x)\rangle>\beta$. We establish sharp directional estimates for problem \eqref{main} in all three regimes $\gamma>1$, $\gamma=1$ and $0<\gamma<1$, throughout the gradient range $0<q\le p$. An exponential change of variable directly removes the natural-growth term when $q=p$ (see \cite{MR3165277,MR4333974}); comparison with that case allows us to treat all $0<q\le p$. The main contribution is the sharp boundary description in the presence of a gradient term, including the borderline case $\gamma=1$.

	Our first result is concerned with the strongly singular case $\gamma>1$.

	\begin{theorem}\label{th:gamma>1} Let $u \in W^{1,p}_{\rm loc}(\Omega) \cap L^\infty(\Omega)$ be a solution to \eqref{main} with $\gamma>1$. Then $u\in C^{1,\alpha}_{\rm loc}(\Omega)\cap C^{\frac{p}{\gamma+p-1}}(\overline{\Omega})$ for some $0<\alpha<1$. Furthermore, the following assertions hold:
		\begin{enumerate}
			\item[(i)] Let $\beta\in[-1,1]$. For any sequences $x_n \in \Omega$ and $\nu_n \in \mathbb{S}^{N-1}$ such that
			\[
			d_\Omega(x_n) \to 0 \quad\text{ and } \quad \langle \nu_n, \eta(x_n) \rangle \to \beta \quad\text{ as } n\to \infty,
			\]
			we have
			\begin{equation}\label{limit>1}
				d_\Omega(x_n)^\frac{\gamma-1}{\gamma+p-1} \frac{\partial u(x_n)}{\partial \nu_n} \to \left(\frac{(\gamma + p - 1)^p}{p^{p-1}(p-1)(\gamma - 1)}\right)^{\frac{1}{\gamma + p - 1}} \frac{p\beta}{\gamma + p - 1}.
			\end{equation}
			\item[(ii)] There exists $C>0$ such that
			\begin{equation}\label{grad_bound>1}
				|\nabla u(x)| \le C d_\Omega(x)^\frac{1-\gamma}{\gamma+p-1} \quad\text{ for all } x\in\Omega.
			\end{equation}
		\end{enumerate}
		Moreover, setting
		\[
		a_\gamma:=\frac{p}{\gamma+p-1},\qquad
		A_\gamma:=\left(\frac{(\gamma+p-1)^p}{p^{p-1}(p-1)(\gamma-1)}\right)^{\frac{1}{\gamma+p-1}},
		\]
		we have the uniform boundary limits
		\begin{equation}\label{vector>1}
		\frac{u(x)}{d_\Omega(x)^{a_\gamma}}\longrightarrow A_\gamma,
		\qquad
		d_\Omega(x)^{1-a_\gamma}\nabla u(x)-a_\gamma A_\gamma\eta(x)\longrightarrow0
		\quad\text{as }d_\Omega(x)\to0.
		\end{equation}
	\end{theorem}
	
	The limit \eqref{limit>1} in Theorem \ref{th:gamma>1} does not depend on $\Omega$, $f$ or the appearance of the gradient term $|\nabla u|^q$.
	A contradiction argument based on \eqref{limit>1} gives the following two-sided bound.
	\begin{corollary}\label{co:gamma>1} Under assumptions of Theorem \ref{th:gamma>1}, for any $\beta\in(0,1)$, there exist $\delta, c_1, c_2>0$ such that
		\[
		c_1 d_\Omega(x)^\frac{1-\gamma}{\gamma+p-1} \le \frac{\partial u(x)}{\partial \nu} \le c_2 d_\Omega(x)^\frac{1-\gamma}{\gamma+p-1} \quad\text{ for all } x\in I_\delta(\Omega)
		\]
		whenever $\nu \in \mathbb{S}^{N-1}$ with $\langle\nu,\eta(x)\rangle\ge\beta$.
	\end{corollary}

	We outline the proof of Theorem \ref{th:gamma>1}. We first establish a sharp estimate of solutions
	\begin{equation}\label{asymp>1}
		c_1 d_\Omega(x)^\frac{p}{\gamma+p-1} \le u(x) \le c_2 d_\Omega(x)^\frac{p}{\gamma+p-1} \quad\text{ for all } x\in\Omega,
	\end{equation}
	where $c_2\ge c_1>0$ (see Theorem \ref{th:asymptotic}).
	This step was done for problem \eqref{singular} in \cite{MR4044739} (see also \cite{MR427826} for the case $p=2$). For problem \eqref{main}, the singular source and the possibly unbounded gradient prevent a direct application of comparison principles requiring bounded gradients near the boundary. For related comparison and qualitative results with regular nonlinearities and first-order terms, see \cite{MR3303939,MR4022661,MR4635360}. We therefore use a comparison principle for the natural-growth term $|\nabla u|^p$, obtained through the classical exponential change of variable recalled above, together with $|\xi|^q\le1+|\xi|^p$ for $0<q\le p$. Once \eqref{asymp>1} is established, rescaling makes the gradient term vanish in the limiting equation, and the half-space classification of \cite{MR4044739} identifies the leading profile. For $\gamma=1$, a separate barrier argument determines the logarithmic coefficient before the rescaled limit is classified by \cite{MR2337497}. Thus the comparison, scaling and classification methods build on existing results; the contribution here is their application to the full gradient range and the resulting sharp boundary asymptotics, including the borderline regime.

	In view of \eqref{asymp>1}, we find that $C^{\frac{p}{\gamma+p-1}}(\overline{\Omega})$ is the optimal boundary H\"older regularity for solutions. Combining Theorem \ref{th:gamma>1} and \eqref{asymp>1} with Theorems \ref{th:gamma=1} and \ref{th:gamma<1} below, we obtain the following.
	\begin{corollary}[Optimal Sobolev integrability]\label{co:W1p0}
Let $u$ be a solution to \eqref{main} and let $1\le r<\infty$.
If $\gamma>1$, then
\[
u\in W^{1,r}_0(\Omega)
\quad\Longleftrightarrow\quad
r<\frac{\gamma+p-1}{\gamma-1}.
\]
In particular, $u\in W^{1,p}_0(\Omega)$ if and only if
$\gamma<(2p-1)/(p-1)$.
If $0<\gamma\le1$, then $u\in W^{1,r}_0(\Omega)$ for every finite $r$;
for $0<\gamma<1$ the gradient is also bounded.
\end{corollary}

	The role of the power $\frac{2p-1}{p-1}$ for problem \eqref{singular} with $p=2$ was explained in \cite{MR3168615}.

	Next, we consider the cases $\gamma=1$ and $0<\gamma<1$, whose boundary profiles differ from the power profile for $\gamma>1$.
	We provide a sharp estimate for $\frac{\partial u(x)}{\partial \nu}$ in the case $\gamma=1$ as follows.
	\begin{theorem}\label{th:gamma=1} Let $u \in W^{1,p}_{\rm loc}(\Omega) \cap L^\infty(\Omega)$ be a solution to \eqref{main} with $\gamma=1$. Then $u\in C^{1,\alpha}_{\rm loc}(\Omega)\cap C^s(\overline{\Omega})$ for some $0<\alpha<1$ and all $s\in(0,1)$. Furthermore, the following assertions hold:
		\begin{enumerate}
			\item[(i)] For any $\beta\in(0,1)$, there exist $0<\delta<1$ and $c_1, c_2>0$ such that
			\begin{equation}\label{bound=1}
				c_1 \left(1-\ln d_\Omega(x)\right)^\frac{1}{p} \le \frac{\partial u(x)}{\partial \nu} \le c_2 \left(1-\ln d_\Omega(x)\right)^\frac{1}{p} \quad\text{ for all } x\in I_\delta(\Omega)
			\end{equation}
			whenever $\nu \in \mathbb{S}^{N-1}$ with $\langle\nu,\eta(x)\rangle\ge\beta$.
			\item[(ii)] There exist $L,C>0$ such that
			\begin{equation}\label{grad_bound=1}
				|\nabla u(x)| \le C \left(L-\ln d_\Omega(x)\right)^\frac{1}{p} \quad\text{ for all } x\in\Omega.
			\end{equation}
			\item[(iii)] Let $\beta\in[-1,1]$. For any sequences $x_n \in \Omega$ and $\nu_n \in \mathbb{S}^{N-1}$ such that
			\[
			d_\Omega(x_n) \to 0 \quad\text{ and } \quad \langle \nu_n, \eta(x_n) \rangle \to \beta \quad\text{ as } n\to \infty,
			\]
			we have
			\begin{equation}\label{limit=1}
				\left(-\ln d_\Omega(x_n)\right)^{-\frac{1}{p}} \frac{\partial u(x_n)}{\partial \nu_n} \to \left(\frac{p}{p-1}\right)^\frac{1}{p} \beta.
			\end{equation}
			Equivalently, with $\kappa_*:=(p/(p-1))^{1/p}$,
\begin{equation}\label{vector=1}
\left(-\ln d_\Omega(x)\right)^{-1/p}\nabla u(x)-\kappa_*\eta(x)
\longrightarrow0\qquad\text{as }d_\Omega(x)\to0,
\end{equation}
uniformly along the boundary. In particular, for each fixed
$\beta_0\in(0,1)$ and $0<\epsilon<\kappa_*$, the constants in
\eqref{bound=1}, for directions satisfying $\nu\cdot\eta\ge\beta_0$,
may be chosen as
$c_1=\beta_0(\kappa_*-\epsilon)$ and $c_2=\kappa_*+\epsilon$
after reducing $\delta$.
		\end{enumerate}
	\end{theorem}
	Our proof reveals that \eqref{asymp>1} does not hold in the case $\gamma=1$. Instead, to prove Theorem \ref{th:gamma=1}, we will derive the following key estimate
	\begin{equation}\label{asymp=1}
		c_1 d_\Omega(x) \left(L-\ln d_\Omega(x)\right)^\frac{1}{p} \le u(x) \le c_2 d_\Omega(x) \left(L-\ln d_\Omega(x)\right)^\frac{1}{p} \quad\text{ for all } x\in \Omega.
	\end{equation}

	The logarithmic profile has a \emph{universal leading coefficient}. This is the content of our next result, which is the exact borderline counterpart of the classification used in Theorem \ref{th:gamma>1} and is the key that allows us to identify the limit \eqref{limit=1}.

	\begin{theorem}[The sharp constant in the borderline case]\label{th:sharp=1}
		Let $u \in W^{1,p}_{\rm loc}(\Omega) \cap L^\infty(\Omega)$ be a solution to \eqref{main} with $\gamma=1$. Then
		\[
		\lim_{d_\Omega(x)\to0^+} \frac{u(x)}{d_\Omega(x)\left(-\ln d_\Omega(x)\right)^\frac{1}{p}} = \left(\frac{p}{p-1}\right)^\frac{1}{p}.
		\]
	\end{theorem}

	\begin{remark}
		Since $\frac{L-\ln t}{-\ln t}\to1$ as $t\to0^+$, the normalizing constant $L$ appearing in \eqref{asymp=1} plays no role in the limit above; any fixed admissible choice of $L$ therefore gives the same limit. Likewise, $(L-\ln t)^{1/p}/(1-\ln t)^{1/p}\to1$, which reconciles the normalizations in \eqref{bound=1} and \eqref{grad_bound=1} after adjusting the constants.
		The constant $\kappa_*:=\left(\frac{p}{p-1}\right)^{1/p}$ is exactly the one dictated by the one-dimensional profile of the equation: if $\psi>0$ solves $-\left(|\psi'|^{p-2}\psi'\right)'=\frac{1}{\psi}$ on $(0,\delta)$ with $\psi(0)=0$ and $\psi'>0$, then multiplying by $\psi'$ and integrating yields the first integral $\left(\psi'(t)\right)^p=\frac{p}{p-1}\left(C-\ln\psi(t)\right)$, whence $\psi(t)\sim\kappa_* t\left(-\ln t\right)^{1/p}$ as $t\to0^+$.
		Thus, exactly as in the case $\gamma>1$, the boundary behavior in the borderline case $\gamma=1$ does not depend on $\Omega$, $f$, $\vartheta$ or $q$. This is consistent with the scaling remark below: for $\gamma=1$ the exponent $\frac{\gamma-1}{\gamma+p-1}$ vanishes, so the rescaling $w(x)=\delta^{-1}u(\delta x)$ leaves the normal derivative unchanged and therefore does not obstruct universality, whereas it does when $0<\gamma<1$.
	\end{remark}

	Now we consider the remaining case $0<\gamma<1$, which may be referred to as the ``weakly singular case''. We show that in this case, solutions to \eqref{main} behave like those of regular problems, namely, the gradient of solutions is bounded near the boundary and the Hopf boundary lemma is valid in the classical sense. More precisely, the estimate
	\[
	c_1 d_\Omega(x) \le u(x) \le c_2 d_\Omega(x) \quad\text{ for all } x\in\Omega
	\]
	holds in this case and
	we have the following result.
	\begin{theorem}\label{th:gamma<1} Let $u \in W^{1,p}_{\rm loc}(\Omega) \cap L^\infty(\Omega)$ be a solution to \eqref{main} with $0<\gamma<1$. Then $u\in C^{1,\alpha}(\overline{\Omega})$ for some $0<\alpha<1$. Moreover, for every $\beta\in(0,1)$, there exists $c>0$ such that
		\begin{equation}\label{hopf<1}
			\frac{\partial u(x)}{\partial \nu}\ge c \quad\text{ for all } x\in \partial\Omega \text{ and } \nu \in \mathbb{S}^{N-1} \text{ with } \langle\nu,\eta(x)\rangle\ge\beta,
		\end{equation}
		where $\eta(x)$ denotes the inward normal of $\Omega$ at $x$.
	\end{theorem}

	\begin{remark}
		Notice that in the case $0<\gamma<1$ we cannot expect an explicit universal value of $\frac{\partial u(x)}{\partial \nu}$ for $x\in \partial\Omega$ as in Theorem \ref{th:gamma>1} or Theorem \ref{th:sharp=1}, because it depends on $\Omega$.
		Indeed, let $u$ be the unique solution to \eqref{singular} with $f\equiv0$. Then $w$ defined by
		\[
		w(x) = \delta^{-\frac{p}{\gamma+p-1}} u(\delta x) \quad \text{ in } \Omega/\delta
		\]
		solves the same equation posed in $\Omega/\delta$, where $\delta>0$, and
		\[
		\frac{\partial w(x)}{\partial \nu} = \delta^{\frac{\gamma-1}{\gamma+p-1}} \frac{\partial u(\delta x)}{\partial \nu} \quad\text{ for all } x\in \partial(\Omega/\delta).
		\]
	\end{remark}

	The Hopf boundary lemma is a fundamental tool in the moving plane method as well as in many other applications. This tool has been used in \cite{MR4044739} to prove the symmetry of solutions to problem \eqref{singular} with $\gamma>1$. In the same manner, we have the following symmetry result for problem \eqref{main} with all $\gamma>0$ in bounded domains that are symmetric and convex in one direction.
	\begin{theorem}\label{th:symmetry}
Let $\gamma>0$. If $N\ge2$ and $\vartheta>0$, assume in addition
$q\ge\max\{p/2,1\}$.
Let $\Omega$ be a bounded $C^2$ domain of $\mathbb{R}^N$ which is strictly
convex in the $x_1$-direction and symmetric with respect to $\{x_1=0\}$.
Let $u$ be a solution of \eqref{main}. Assume that the total source satisfies
\begin{equation}\label{positive-source}
F(s):=s^{-\gamma}+f(s)>0
\qquad\text{for }0<s\le\|u\|_{L^\infty(\Omega)}.
\end{equation}
Then $u$ is symmetric with respect to $\{x_1=0\}$ and nondecreasing
in the $x_1$-direction in $\Omega\cap\{x_1<0\}$.
In particular, in a ball $u$ is radially symmetric and decreasing
as a function of the radius.
\end{theorem}
Condition \eqref{positive-source} holds, in particular, when $f\ge0$;
it also allows perturbations of either sign. When $\vartheta=0$, the exponent
$q$ is irrelevant. Here strict directional convexity is understood in the usual
moving-plane sense: sections parallel to $e_1$ are intervals, and the inward
normal satisfies $\eta\cdot e_1>0$ on $\partial\Omega\cap\{x_1<0\}$.
When $N=1$, the conclusion holds throughout $0<q\le p$; the additional
lower bound on $q$ is not needed in the one-dimensional proof below.

	The rest of this paper is organized as follows. In Section \ref{sect2}, we establish boundary continuity, two-sided bounds, and the sharp logarithmic coefficient in Theorem \ref{th:sharp=1}. In Section \ref{sect3}, we prove our main result, namely, Theorems \ref{th:gamma>1}, \ref{th:gamma=1} and \ref{th:gamma<1} on the behaviors of the derivatives of solutions near the boundary. In Section \ref{sect4}, we apply previous results to show the symmetry of solutions stated in Theorem \ref{th:symmetry}.

	\section{Sharp bounds for solutions}\label{sect2}
	We aim to prove the following result in this section.
	\begin{theorem}\label{th:asymptotic}
		Let $u \in W^{1,p}_{\rm loc}(\Omega) \cap L^\infty(\Omega)$ be a solution to \eqref{main} with $\gamma>0$. Then $u\in C^{1,\alpha}_{\rm loc}(\Omega)\cap C(\overline{\Omega})$ for some $0<\alpha<1$. Moreover, the following assertions hold:
		\begin{enumerate}
			\item[(i)] When $\gamma>1$, there exist $c_1,c_2>0$ such that
			\[
			c_1 d_\Omega(x)^\frac{p}{\gamma+p-1} \le u(x) \le c_2 d_\Omega(x)^\frac{p}{\gamma+p-1} \quad\text{ for } x\in \Omega.
			\]
			\item[(ii)] When $\gamma=1$, there exist $L,c_1,c_2>0$ such that
			\[
			c_1 d_\Omega(x) \left(L-\ln d_\Omega(x)\right)^\frac{1}{p} \le u(x) \le c_2 d_\Omega(x) \left(L-\ln d_\Omega(x)\right)^\frac{1}{p} \quad\text{ for } x\in \Omega.
			\]
			\item[(iii)] When $0<\gamma<1$, there exist $c_1,c_2>0$ such that
			\[
			c_1 d_\Omega(x) \le u(x) \le c_2 d_\Omega(x) \quad\text{ for } x\in \Omega.
			\]
		\end{enumerate}
	\end{theorem}

	More precisely, set $D_\Omega=\max_{\overline\Omega}d_\Omega>0$. In (ii) one may choose any $L>\ln D_\Omega$; the constants $c_1,c_2>0$ then depend on $L$.

\begin{lemma}[Boundary continuity]\label{lem:continuity}
Every solution in the weak class specified in the introduction belongs to
$C^{1,\alpha}_{\rm loc}(\Omega)\cap C(\overline\Omega)$ for some $\alpha\in(0,1)$,
and its continuous boundary value is zero.
\end{lemma}
\begin{proof}
On each compact subset of $\Omega$, positivity and boundedness of $u$ give
a bound for $u^{-\gamma}+f(u)$. The lower-order term has natural growth,
since $|\xi|^q\le1+|\xi|^p$. Interior regularity for bounded weak solutions
of scalar quasilinear equations therefore gives $u\in C^{1,\alpha}_{\rm loc}(\Omega)$;
see \cite{MR709038,MR727034}.

Fix $\epsilon>0$ and put $v_\epsilon=(u-\epsilon)^+\in W^{1,p}_0(\Omega)$.
Let $B=\max_{0\le s\le\|u\|_\infty}|f(s)|$ and $M_\epsilon=\epsilon^{-\gamma}+B$.
Testing the equation with an increasing smooth approximation of
$\mathbf{1}_{\{u>\epsilon\}}$ multiplied by a nonnegative compactly supported
test function, and discarding the nonnegative truncation and absorption terms,
gives
\[
-\Delta_p v_\epsilon\le M_\epsilon\quad\text{weakly in }\Omega.
\]
Let $z\in W^{1,p}_0(\Omega)\cap C(\overline\Omega)$ solve
$-\Delta_p z=1$ in $\Omega$, $z=0$ on $\partial\Omega$.
The ordinary Sobolev comparison principle yields
\[
0\le v_\epsilon\le M_\epsilon^{1/(p-1)}z,
\qquad
u\le\epsilon+M_\epsilon^{1/(p-1)}z.
\]
Consequently $\limsup_{d_\Omega(x)\to0}u(x)\le\epsilon$.
Letting $\epsilon\downarrow0$ proves the claimed uniform boundary continuity.
\end{proof}

	The proof of Theorem \ref{th:asymptotic} exploits various types of weak comparison principles.
	The first one we require in our later arguments is the following.
	\begin{lemma}[Weak comparison without a gradient term; cf.\ \cite{MR4044739}]\label{lem:wcp}
		Let $c>0$.
		If $u_1,u_2\in C^{1,\alpha}_{\rm loc}(\Omega) \cap C(\overline\Omega)$ satisfy
		\[
		\begin{cases}
			-\Delta_p u_1 \le \frac{c}{u_1^\gamma} &\text{ in } \Omega,\\
			-\Delta_p u_2 \ge \frac{c}{u_2^\gamma} &\text{ in } \Omega,\\
			u_1,u_2>0 &\text{ in } \Omega,\\
			u_1 \le u_2 &\text{ on } \partial\Omega,
		\end{cases}
		\]
		then $u_1 \le u_2$ in $\Omega$.
	\end{lemma}
	\begin{proof}
For $\epsilon>0$, continuity up to the boundary implies that
$w=(u_1-u_2-\epsilon)^+$ has compact support in $\Omega$.
Testing the two inequalities with $w$ and subtracting gives
\[
\int_\Omega\bigl(|\nabla u_1|^{p-2}\nabla u_1
-|\nabla u_2|^{p-2}\nabla u_2\bigr)\cdot\nabla w
\le c\int_\Omega(u_1^{-\gamma}-u_2^{-\gamma})w\le0.
\]
Strict monotonicity of $\xi\mapsto|\xi|^{p-2}\xi$ and the zero trace
of $w$ imply $w=0$. Let $\epsilon\downarrow0$.
\end{proof}
All boundary inequalities in the comparison lemmas are now pointwise.
The same arguments apply on any bounded open subdomain where the compared
functions are continuous up to its boundary. This lemma is useful for problem \eqref{singular}. However, to deal with the gradient term $|\nabla u|^q$ in problem \eqref{main}, the principle stated in Lemma \ref{lem:wcp} is not enough.
	We next establish a weak comparison principle with natural gradient growth.

	\begin{lemma}[Weak comparison principle]\label{lem:wcp_singular}
		Let $c>0$ and let $u_1,u_2\in C^{1,\alpha}_{\rm loc}(\Omega) \cap C(\overline\Omega)$ be such that
		\[
		\begin{cases}
			-\Delta_p u_1 + \vartheta |\nabla u_1|^p + \vartheta \le \frac{c}{u_1^\gamma} &\text{ in } \Omega,\\
			-\Delta_p u_2 + \vartheta |\nabla u_2|^p + \vartheta \ge \frac{c}{u_2^\gamma} &\text{ in } \Omega,\\
			u_1,u_2 > 0 &\text{ in } \Omega,\\
			u_1 \le u_2 &\text{ on } \partial \Omega.
		\end{cases}
		\]
		If $\vartheta>0$, we further assume
		\[
		u_1 \le \left( \frac{c}{\vartheta} \right)^{1/\gamma} \quad\text{ in } \Omega.
		\]
		Then $u_1 \le u_2$ in $ \Omega$.
	\end{lemma}

	\begin{proof}
		Since the case $\vartheta=0$ was studied in Lemma \ref{lem:wcp}, we will assume $\vartheta>0$ in what follows.
		The difficulty lies in the gradient terms, which may be unbounded near the boundary. We eliminate these terms by employing the following change of variable
		\[
		v_i=e^{-\frac{\vartheta u_i}{p-1}}, \quad i=1,2.
		\]
		Then $v_1,v_2\in C^{1,\alpha}_{\rm loc}(\Omega)$ and
		\[
		e^{-\frac{\vartheta}{p-1} \left( \frac{c}{\vartheta} \right)^{1/\gamma}} \le v_1<1,\quad
		e^{-\frac{\vartheta \|u_2\|_{L^\infty(\Omega)}}{p-1}} \le v_2 < 1 \quad\text{ in } \Omega.
		\]
		In the weak sense, we have
		\begin{align*}
			\Delta_p v_1 &= \left(\frac{\vartheta v_1}{p-1}\right)^{p-1} (-\Delta_p u_1 + \vartheta |\nabla u_1|^p)\\
			&\le \left(\frac{\vartheta v_1}{p-1}\right)^{p-1} \left(\frac{c}{u_1^\gamma} - \vartheta\right)\\
			&= \left(\frac{\vartheta v_1}{p-1}\right)^{p-1} \left(\frac{c}{(-\frac{p-1}{\vartheta}\ln v_1)^\gamma} - \vartheta\right) \quad\text{ in } \Omega.
		\end{align*}
		Similarly,
		\[
		\Delta_p v_2 \ge \left(\frac{\vartheta v_2}{p-1}\right)^{p-1} \left(\frac{c}{(-\frac{p-1}{\vartheta}\ln v_2)^\gamma} - \vartheta\right) \quad\text{ in } \Omega.
		\]

		Therefore, $v_1,v_2$ solve
		\begin{equation}\label{v_sys}
			\begin{cases}
				\Delta_p v_1 \le g(v_1) &\text{ in } \Omega,\\
				\Delta_p v_2 \ge g(v_2) &\text{ in } \Omega,\\
				0<v_1,v_2<1 &\text{ in } \Omega,\\
				v_2 \le v_1 &\text{ on } \partial\Omega,
			\end{cases}
		\end{equation}
		where
		\[
		g(t) = \left(\frac{\vartheta t}{p-1}\right)^{p-1} \left(\frac{c}{(-\frac{p-1}{\vartheta}\ln t)^\gamma} - \vartheta\right), \quad t\in(0,1).
		\]

		For $\varepsilon>0$, continuity of $v_1,v_2$ up to the boundary
and $v_2\le v_1$ there imply
$\operatorname{supp}(v_2-v_1-\varepsilon)^+\subset\subset\Omega$.
Thus $w=(v_2-v_1-\varepsilon)^+$ is an admissible test function;
on its support both functions stay in a compact subset of $(0,1)$. Then subtracting, we derive
		\[
		\int_\Omega \langle |\nabla v_2|^{p-2}\nabla v_2 - |\nabla v_1|^{p-2}\nabla v_1, \nabla w \rangle \le \int_\Omega \left(g(v_1) - g(v_2)\right) w.
		\]

		We recall the following well-known elementary inequality: there exists a positive constant \( \Lambda \) depending only on \( p > 1 \) such that
		\begin{equation}\label{p_ineq}
			\langle |\xi|^{p-2} \xi - |\xi'|^{p-2} \xi', \, \xi - \xi' \rangle \geq \Lambda (|\xi| + |\xi'|)^{p-2} |\xi - \xi'|^2
		\end{equation}
		for all $\xi, \xi' \in \mathbb{R}^N$ with $|\xi| + |\xi'| > 0$.
		In view of \eqref{p_ineq}, we have
		\begin{equation}\label{wcp_key}
			\Lambda \int_\Omega (|\nabla v_1| + |\nabla v_2|)^{p-2} |\nabla w|^2 \le \int_\Omega \left(g(v_1) - g(v_2)\right) w.
		\end{equation}

		On the other hand,
		\[
		g'(t) = \frac{\vartheta^{p-1} t^{p-2}}{(p-1)^{p-2}} \left[ \left( \frac{c}{\left( -\frac{p-1}{\vartheta} \ln t \right)^\gamma} - \vartheta \right) + \frac{c \gamma}{\vartheta\left( -\frac{p-1}{\vartheta} \ln t \right)^{\gamma + 1}} \right]
		\quad \text{for } t \in (0,1).
		\]
		Hence $g'(t) > 0$ whenever $g(t)\ge0$. Moreover,
		\[
		g(t) \ge 0 \iff e^{-\frac{\vartheta}{p-1} \left( \frac{c}{\vartheta} \right)^{1/\gamma}} \le t < 1.
		\]
		Thus, $g$ is increasing on $\left[e^{-\frac{\vartheta}{p-1} \left( \frac{c}{\vartheta} \right)^{1/\gamma}}, 1\right)$.
		Notice now that in $\Omega\cap\{w>0\}$, we have
		\[
		e^{-\frac{\vartheta}{p-1} \left( \frac{c}{\vartheta} \right)^{1/\gamma}} \le v_1<v_1+\varepsilon < v_2 < 1.
		\]
		Hence
		\[
		g(v_1) - g(v_2) < 0.
		\]
		Therefore, \eqref{wcp_key} yields
		\[
		\int_\Omega (|\nabla v_1| + |\nabla v_2|)^{p-2} |\nabla w|^2 = 0.
		\]
		Consequently, $v_2 \le v_1 + \varepsilon$ in $ \Omega$. Since $\varepsilon$ is arbitrary, we get $v_2 \le v_1$ in $ \Omega$. This means $u_1 \le u_2$ in $ \Omega$.
	\end{proof}

	We prepare the following lemma for later use in dealing with the case $0<\gamma<1$.
	\begin{lemma}\label{lem:U}
		Assume $0<\gamma<1$. Then for every $M>0$ and $r_2>r_1>0$, the problem
		\begin{equation}\label{U}
			\begin{cases}
				-\Delta_p u = \frac{M}{u^\gamma} &\text{ in } B_{r_2}\setminus\overline{B_{r_1}},\\
				u>0 &\text{ in } B_{r_2}\setminus\overline{B_{r_1}},\\
				u=0 &\text{ on } \partial(B_{r_2}\setminus\overline{B_{r_1}})
			\end{cases}
		\end{equation}
		admits a unique solution $U \in C^{1,\alpha}(B_{r_2}\setminus\overline{B_{r_1}}) \cap C(\overline{B_{r_2}}\setminus B_{r_1})$. Moreover, there exist $c,C>0$ such that
		\[
		c d_{B_{r_2}\setminus\overline{B_{r_1}}}(x) \le U(x) \le C d_{B_{r_2}\setminus\overline{B_{r_1}}}(x) \quad\text{ for all } x \in B_{r_2}\setminus\overline{B_{r_1}}.
		\]
	\end{lemma}
	\begin{proof}
		The existence of a unique solution $U$ to \eqref{U} is proved in \cite{MR3136107,MR3478284}. Hence $U$ is radial. By abuse of notation, we also write $U(x)=U(r)$, where $r=|x|$. We claim that, for some $c>0$,
		\begin{equation}\label{U_lower_bound}
			U(r) \ge c \min\{r-r_1, r_2-r\} \quad\text{ for all } r \in (r_1,r_2).
		\end{equation}

		To prove this, let $\lambda_1>0$ be the first eigenvalue and $\phi_1 \in C^1(\overline{B_{r_2}}\setminus{B_{r_1}})$ a corresponding positive eigenfunction of the $p$-Laplacian in $B_{r_2}\setminus\overline{B_{r_1}}$ with the Dirichlet boundary condition. Then we set
		\[
		V_s = s \phi_1,
		\]
		where $s>0$. We have
		\[
		-\Delta_p V_s = \frac{h_s(x)}{V_s^\gamma},
		\]
		where
		\[
		h_s(x) = \lambda_1 s^{\gamma+p-1} \phi_1^{\gamma+p-1}.
		\]

		Let $\tilde s>0$ be small such that
		\[
		h_{\tilde s}(x) = \lambda_1 \tilde s^{\gamma+p-1}\phi_1(x)^{\gamma+p-1} \le M \quad\text{ for } x\in B_{r_2}\setminus\overline{B_{r_1}}.
		\]

		We have now
		\[
		\begin{cases}
			-\Delta_p U = \frac{M}{U^\gamma} &\text{ in } B_{r_2}\setminus\overline{B_{r_1}},\\
			- \Delta_p V_{\tilde s} \le \frac{M}{V_{\tilde s}^\gamma} &\text{ in } B_{r_2}\setminus\overline{B_{r_1}},\\
			U,V_{\tilde s}>0 &\text{ in } B_{r_2}\setminus\overline{B_{r_1}},\\
			U = V_{\tilde s} = 0 &\text{ on } \partial (B_{r_2}\setminus\overline{B_{r_1}}).
		\end{cases}
		\]
		Hence Lemma \ref{lem:wcp} gives
		\[
		V_{\tilde s}\le U \quad \text{ in } B_{r_2}\setminus\overline{B_{r_1}}.
		\]
		This implies
		\[
		U(x) \ge c d_{B_{r_2}\setminus\overline{B_{r_1}}}(x) \quad \text{ in } B_{r_2}\setminus\overline{B_{r_1}}
		\]
		for some $c>0$. Hence \eqref{U_lower_bound} is verified.

		Since $U$ is radial, exploiting \eqref{U} and \eqref{U_lower_bound}, we get
		\[
		-\frac{1}{r^{N-1}} (r^{N-1} |U'(r)|^{p-2} U'(r))' = \frac{M}{U(r)^\gamma} \le \frac{M}{c^\gamma(r-r_1)^\gamma} \quad\text{ for } r_1 < r \le \frac{r_1+r_2}{2}.
		\]
		The right-hand side is integrable up to $r_1$ because $\gamma<1$. The radial flux therefore has a finite limit at $r_1$, and $U' $ extends continuously there. Integrating, we deduce
		\begin{align*}
			0 &\le r_1^{N-1} |U'(r_1)|^{p-2} U'(r_1)\\
			&\le \int_{r_1}^{\frac{r_1+r_2}{2}} \frac{M r^{N-1}}{c^\gamma(r-r_1)^\gamma} dr + \left(\frac{r_1+r_2}{2}\right)^{N-1} \left|U'\left(\frac{r_1+r_2}{2}\right)\right|^{p-2} U'\left(\frac{r_1+r_2}{2}\right)\\
			& < +\infty,
		\end{align*}
		thanks to the fact $\gamma<1$. Thus there exists $C>0$ such that
		\[
		U(r) \le C (r-r_1) \quad\text{ for all } r \in (r_1,r_2).
		\]

		Similarly, by integrating on $[r,r_2] \subset \left[\frac{r_1+r_2}{2}, r_2\right]$, we have
		for some $C'>0$,
		\[
		U(r) \le C' (r_2-r) \quad\text{ for all } r \in (r_1,r_2).
		\]

		The radial flux has H\"older continuous extensions to both endpoints with exponent $1-\gamma$. Inverting $t\mapsto|t|^{p-2}t$ then gives $U\in C^{1,\alpha}(\overline{B_{r_2}}\setminus B_{r_1})$ for some $\alpha>0$. The lemma is proved.
	\end{proof}

	We are ready to prove Theorem \ref{th:asymptotic}. In the following choices of barriers, a condition involving division by $\vartheta$ is imposed only when $\vartheta>0$.

	\begin{proof}[Proof of Theorem \ref{th:asymptotic}]
		Lemma \ref{lem:continuity} already gives interior $C^{1,\alpha}$ regularity and continuity with zero boundary value. We now prove the quantitative bounds.

		Let $\lambda_1>0$ and $\phi_1 \in C^1(\overline{\Omega})$ be the first eigenvalue and a corresponding positive eigenfunction of the $p$-Laplacian in $\Omega$ with the Dirichlet boundary condition, namely,
		\[
		\begin{cases}
			-\Delta_p \phi_1 = \lambda_1 \phi_1^{p-1} & \text{ in } \Omega,\\
			\phi_1 > 0 & \text{ in } \Omega,\\
			\phi_1 = 0 & \text{ on } \partial \Omega.
		\end{cases}
		\]
		We consider three cases.

		\textit{Case 1: $\gamma>1$.}
		Setting
		\[
		w_s = s \phi_1^\frac{p}{\gamma+p-1},
		\]
		where $s>0$. A direct calculation yields
		\[
		-\Delta_p w_s = \frac{a_s(x)}{w_s^\gamma},
		\]
		where
		\[
		a_s(x) = s^{\gamma+p-1} \left(\frac{p}{\gamma+p-1}\right)^{p-1} \left[\lambda_1\phi_1(x)^p + \frac{(\gamma-1)(p-1)}{\gamma+p-1}|\nabla\phi_1(x)|^p\right].
		\]

		Since $u \in L^\infty(\Omega)$, we can find $M>0$ sufficiently large such that
		\begin{equation}\label{M}
			\frac{1}{u^\gamma} + f(u) < \frac{M}{u^\gamma} \quad\text{ for } x\in \Omega.
		\end{equation}
		On the other hand, by the positivity of $\phi_1$ and the Hopf boundary lemma for $\phi_1$, we know that
		\[
		\inf_{x\in \Omega} \left[\lambda_1\phi_1(x)^p + \frac{(\gamma-1)(p-1)}{\gamma+p-1}|\nabla\phi_1(x)|^p\right] > 0.
		\]
		Hence, we can choose some large $s_1>0$ such that
		\[
		a_{s_1}(x) > M \quad\text{ for } x\in \Omega.
		\]

		We have now
		\[
		\begin{cases}
			-\Delta_p u \le \frac{M}{u^\gamma} &\text{ in } \Omega,\\
			-\Delta_p w_{s_1} \ge \frac{M}{w_{s_1}^\gamma} &\text{ in } \Omega,\\
			w_{s_1}, u > 0 &\text{ in } \Omega,\\
			w_{s_1} = u = 0 &\text{ on } \partial \Omega.
		\end{cases}
		\]

		Lemma \ref{lem:wcp} can be applied to yield
		\[
		u \le w_{s_1} \quad \text{ in } \Omega.
		\]
		Hence
		\[
		u \le C d_\Omega(x)^\frac{p}{\gamma+p-1} \quad \text{ in } \Omega \quad\text{ for some } C>0.
		\]

		Next, we prove the lower bound. To this end, we write
		\[
		-\Delta_p w_s + \vartheta |\nabla w_s|^p + \vartheta = \frac{b_s(x)}{w_s^\gamma},
		\]
		where, with $a_\gamma=p/(\gamma+p-1)$ as in Theorem \ref{th:gamma>1},
		\begin{align*}
			b_s(x)
			&= a_s(x) + \vartheta (|\nabla w_s|^p + 1)w_s^\gamma\\
				&= a_s(x)+\vartheta s^{\gamma+p}a_\gamma^p
                   \phi_1^{a_\gamma}|\nabla\phi_1|^p
                   +\vartheta s^\gamma\phi_1^{\gamma a_\gamma}.
		\end{align*}

		Exploiting $u \in C(\overline{\Omega})$ (Lemma \ref{lem:continuity}) and the Dirichlet boundary condition, we can find $\delta>0$ sufficiently small such that
		\[
		\frac{1}{u^\gamma} + f(u) > \frac{1}{2u^\gamma} \quad\text{ for } x\in I_\delta(\Omega).
		\]
		We also choose some small $s_2>0$ such that
		\[
		w_{s_2}(x) \le \left( \frac{1}{2\vartheta} \right)^{1/\gamma},\quad b_{s_2}(x) < \frac{1}{2} \quad\text{ for } x\in I_\delta(\Omega)
		\]
		and
		\[
		w_{s_2}(x) \le u(x) \quad\text{ for } x\in \partial I_\delta(\Omega) \setminus \partial\Omega.
		\]

		Since $0<q\le p$, it follows that $|\nabla u|^q \le |\nabla u|^p + 1$. Therefore, we have
		\[
		\begin{cases}
			-\Delta_p w_{s_2} + \vartheta |\nabla w_{s_2}|^p + \vartheta \le \frac{1}{2 w_{s_2}^\gamma} &\text{ in } I_\delta(\Omega),\\
			-\Delta_p u + \vartheta |\nabla u|^p + \vartheta \ge \frac{1}{2 u^\gamma} &\text{ in } I_\delta(\Omega),\\
			w_{s_2}, u > 0 &\text{ in } I_\delta(\Omega),\\
			w_{s_2} \le u &\text{ on } \partial I_\delta(\Omega),
		\end{cases}
		\]

		Lemma \ref{lem:wcp_singular} can be applied now to yield
		\[
		w_{s_2}\le u \quad \text{ in } I_\delta(\Omega).
		\]
		By compactness arguments, this implies
		\[
		u \ge c d_\Omega(x)^\frac{p}{\gamma+p-1} \quad \text{ in } \Omega \quad\text{ for some } c>0.
		\]

		\textit{Case 2: $0<\gamma<1$.} This case is also treated via comparison principles, but with a different choice of barrier functions. The upper barrier based on the power of the first eigenfunction does not give the required linear boundary behavior. Instead, we use the one from Lemma \ref{lem:U}.

		Let $M>0$ be as in \eqref{M}. Now we fix $r_1>0$ small such that for every $x_0 \in \partial \Omega$, there is a ball $B_{r_1}(y_0) \subset\mathbb{R}^N\setminus\overline\Omega$ that is tangent to $\partial\Omega$ at $x_0$. We also fix $r_2=\text{diam}(\Omega)+r_1$, where
		\[
		\text{diam}(\Omega) = \sup_{x_1,x_2\in\Omega} |x_1-x_2|
		\]
		is the diameter of $\Omega$. Let $U$ be defined as in Lemma \ref{lem:U}.
		Consider any $x \in \Omega$ and let $x_0 \in \partial \Omega$ be such that $d_\Omega(x) = |x-x_0|$. Let $B_{r_1}(y_0) \subset\mathbb{R}^N\setminus\overline\Omega$ be the ball of radius $r_1$ that touches $\partial\Omega$ at $x_0$, then
		\[
		\Omega \subset B_{r_2}(y_0)\setminus\overline{B_{r_1}(y_0)}.
		\]
		Moreover,
		\[
		\begin{cases}
			-\Delta_p u \le \frac{M}{u^\gamma} &\text{ in } \Omega,\\
			- \Delta_p U_{y_0} = \frac{M}{U_{y_0}^\gamma} &\text{ in } \Omega,\\
			u,U_{y_0}>0 &\text{ in } \Omega,\\
			0 = u \le U_{y_0} &\text{ on } \partial \Omega,
		\end{cases}
		\]
		where $U_{y_0} (x) := U(x - y_0)$.
		Hence by Lemma \ref{lem:wcp}, we deduce
		\[
		u \le U_{y_0} \quad \text{ in } \Omega.
		\]
		Now Lemma \ref{lem:U} gives
		\[
		u(x) \le C\bigl(|x-y_0|-r_1\bigr) \le C d_\Omega(x) \quad \text{ in } \Omega.
		\]

		To derive the lower bound, we set
		\[
		v_s = s \phi_1,
		\]
		where $s>0$ and we recall that $\phi_1$ is a first positive eigenfunction of the $p$-Laplacian in $\Omega$. We have
		\[
		-\Delta_p v_s + \vartheta |\nabla v_s|^p + \vartheta = \frac{e_s(x)}{v_s^\gamma},
		\]
		where
		\[
		e_s(x) = s^{p + \gamma - 1} \lambda_1 \phi_1^{p + \gamma - 1}
		+ \vartheta s^{p + \gamma} \phi_1^\gamma |\nabla \phi_1|^p
		+ \vartheta s^\gamma \phi_1^\gamma.
		\]

		Since $u \in C(\overline{\Omega})$ and $u=0$ on $\partial\Omega$, we can find $\delta>0$ sufficiently small such that
		\[
		\frac{1}{u^\gamma} + f(u) > \frac{1}{2u^\gamma} \quad\text{ for } x\in I_\delta(\Omega).
		\]
		We also choose some small $s_0>0$ such that
		\[
		v_{s_0}(x) \le \left( \frac{1}{2\vartheta} \right)^{1/\gamma},\quad e_{s_0}(x) < \frac{1}{2} \quad\text{ for } x\in I_\delta(\Omega)
		\]
		and
		\[
		v_{s_0}(x) \le u(x) \quad\text{ for } x\in \partial I_\delta(\Omega) \setminus \partial\Omega.
		\]

		We have now
		\[
		\begin{cases}
			-\Delta_p u + \vartheta |\nabla u|^p + \vartheta \ge \frac{1}{2u^\gamma} &\text{ in } I_\delta(\Omega),\\
			- \Delta_p v_{s_0} + \vartheta |\nabla v_{s_0}|^p + \vartheta \le \frac{1}{2v_{s_0}^\gamma} &\text{ in } I_\delta(\Omega),\\
			u,v_{s_0}>0 &\text{ in } I_\delta(\Omega),\\
			v_{s_0} \le u &\text{ on } \partial I_\delta(\Omega).
		\end{cases}
		\]
		Hence Lemma \ref{lem:wcp_singular} gives
		\[
		v_{s_0}\le u \quad \text{ in } I_\delta(\Omega).
		\]

		From this and the compactness argument, we obtain
		\[
		u (x)\ge c d_\Omega(x) \quad \text{ in } \Omega
		\]
		for some $c>0$.

		\textit{Case 3: $\gamma=1$.}
		Setting
		\[
		w_s = s \phi_1 (1- \ln \phi_1)^\frac{1}{p},
		\]
		where $s>0$ and $\phi_1$ is the first positive eigenfunction of the $p$-Laplacian in $\Omega$ such that $\|\phi_1\|_{L^\infty(\Omega)}=1$. After tedious computations, we obtain
		\[
		-\Delta_p w_s = \frac{a_s(x)}{w_s},
		\]
		where
		\begin{align*}
			&a_s(x)\\
			&= s^p \left[z\left(z - \frac{1}{p} z^{1-p}\right)^{p-1} \lambda_1 \phi_1^p + \frac{p-1}{p} \left(1 - \frac{1}{p} z^{-p}\right)^{p-2} \left(1 + \frac{p-1}{p} z^{-p}\right) |\nabla\phi_1|^p\right]
		\end{align*}
		with $z:=(1- \ln \phi_1)^\frac{1}{p}$. The range of $z$ is $[1,+\infty)$. Moreover, by l'H\^opital's rule, we have $\phi_1 z^k \to 0$ as $d_\Omega(x)\to0$ for any $k>0$. Hence, both terms in the square bracket are bounded above in $\Omega$. Their sum is also bounded below by a positive constant.

		Let $M>0$ as in \eqref{M} and choose some large $s_1>0$ such that
		\[
		a_{s_1}(x) > M \quad\text{ for } x\in \Omega.
		\]

		We have now
		\[
		\begin{cases}
			-\Delta_p u \le \frac{M}{u} &\text{ in } \Omega,\\
			-\Delta_p w_{s_1} \ge \frac{M}{w_{s_1}} &\text{ in } \Omega,\\
			w_{s_1}, u > 0 &\text{ in } \Omega,\\
			w_{s_1} = u = 0 &\text{ on } \partial \Omega.
		\end{cases}
		\]

		Lemma \ref{lem:wcp} can be applied to yield
		\[
		u \le w_{s_1} \quad \text{ in } \Omega.
		\]
		Since $t\mapsto t(1-\ln t)^{1/p}$ is increasing on $(0,1]$ and $\phi_1\le\min\{1,C d_\Omega\}$, this gives
		\[
		u \le C d_\Omega(x) \left(L-\ln d_\Omega(x)\right)^\frac{1}{p} \quad \text{ in } \Omega \quad\text{ for some } C, L>0.
		\]

		Next, we prove the lower bound. To this end, we write
		\[
		-\Delta_p w_s + \vartheta |\nabla w_s|^p + \vartheta = \frac{b_s(x)}{w_s},
		\]
		where
		\begin{align*}
			&b_s(x)\\
			&= a_s(x) + \vartheta (|\nabla w_s|^p + 1)w_s\\
			&= s^p \left[z\left(z - \frac{1}{p} z^{1-p}\right)^{p-1} \lambda_1 \phi_1^p + \frac{p-1}{p} \left(1 - \frac{1}{p} z^{-p}\right)^{p-2} \left(1 + \frac{p-1}{p} z^{-p}\right) |\nabla\phi_1|^p\right]\\
			&\quad+ \vartheta s^{p+1} \phi_1 z \left(z - \frac{1}{p} z^{1-p}\right)^p |\nabla \phi_1|^p + \vartheta s\phi_1 z.
		\end{align*}

		In particular, $b_s\to0$ uniformly as $s\downarrow0$. Exploiting continuity and the zero boundary value, we find $\delta>0$ sufficiently small such that
		\[
		\frac{1}{u} + f(u) > \frac{1}{2u} \quad\text{ for } x\in I_\delta(\Omega).
		\]
		We also choose some small $s_2>0$ such that
		\[
		w_{s_2}(x) \le \frac{1}{2\vartheta},\quad b_{s_2}(x) < \frac{1}{2} \quad\text{ for } x\in I_\delta(\Omega)
		\]
		and
		\[
		w_{s_2}(x) \le u(x) \quad\text{ for } x\in \partial I_\delta(\Omega) \setminus \partial\Omega.
		\]

		We have now
		\[
		\begin{cases}
			-\Delta_p w_{s_2} + \vartheta |\nabla w_{s_2}|^p + \vartheta \le \frac{1}{2 w_{s_2}} &\text{ in } I_\delta(\Omega),\\
			-\Delta_p u + \vartheta |\nabla u|^p + \vartheta \ge \frac{1}{2 u} &\text{ in } I_\delta(\Omega),\\
			w_{s_2}, u > 0 &\text{ in } I_\delta(\Omega),\\
			w_{s_2} \le u &\text{ on } \partial I_\delta(\Omega),
		\end{cases}
		\]

		Lemma \ref{lem:wcp_singular} can be applied now to yield
		\[
		w_{s_2}\le u \quad \text{ in } I_\delta(\Omega).
		\]
		By compactness arguments, this implies
		\[
		u \ge c d_\Omega(x) \left(L-\ln d_\Omega(x)\right)^\frac{1}{p} \quad \text{ in } \Omega \quad\text{ for some } c>0.
		\]

		This concludes the proof of Theorem \ref{th:asymptotic}.
	\end{proof}

We now turn to the sharp constant in the borderline case.
We prove Theorem \ref{th:sharp=1} using barriers built directly from the
distance function. This keeps the comparison on a fixed boundary collar.

\begin{lemma}[Logarithmic barriers]\label{lem:log_barrier}
For $\kappa>0$ and $K\in\mathbb R$, set
\[
h(t)=\kappa t(K-\ln t)^{1/p},\qquad z=(K-\ln t)^{1/p}.
\]
Whenever $K-\ln t>1/p$, we have
\[
h'=\kappa\left(z-\tfrac1p z^{1-p}\right)>0,
\qquad
h''=-\frac{\kappa}{pt}\left(z^{1-p}+\tfrac{p-1}{p}z^{1-2p}\right),
\]
and
\begin{equation}\label{barrier_1d}
\begin{gathered}
-(p-1)(h')^{p-2}h''=\frac{(p-1)\kappa^p}{p h}\Theta(z),
\\
\Theta(z)=\left(1-\tfrac1p z^{-p}\right)^{p-2}
\left(1+\tfrac{p-1}{p}z^{-p}\right)\longrightarrow1.
\end{gathered}
\end{equation}
In a collar where $d=d_\Omega$ is $C^2$, the function $V=h(d)$ satisfies
\begin{align}
(-\Delta_p V)V
&=\frac{(p-1)\kappa^p}{p}\Theta(z)
-\kappa^p d z^p\left(1-\tfrac1p z^{-p}\right)^{p-1}\Delta d,
\label{barrier_distance}\\
|\nabla V|^p V
&=\kappa^{p+1}d z^{p+1}\left(1-\tfrac1p z^{-p}\right)^p.
\label{barrier_grad}
\end{align}
Here $z=(K-\ln d)^{1/p}$.
\end{lemma}
\begin{proof}
Differentiate $h$ and use $|\nabla d|=1$. More explicitly,
\begin{align*}
-(p-1)(h')^{p-2}h''
&=\frac{(p-1)\kappa^{p-1}}{ptz}\Theta(z),\\
-\Delta_p h(d)&=-(p-1)(h')^{p-2}h''-(h')^{p-1}\Delta d.
\end{align*}
Multiplication by $h(d)=\kappa d z$ gives the claims.
\end{proof}

\begin{proof}[Proof of Theorem \ref{th:sharp=1}]
Write $\kappa_*=(p/(p-1))^{1/p}$ and fix $0<\epsilon<1/4$.
Choose $\delta_0>0$ so that $d_\Omega\in C^2(I_{\delta_0}(\Omega))$
and $|\Delta d_\Omega|\le B$ there. By reducing $\delta_0$, continuity of
$u$ and boundedness of $f$ give
\begin{equation}\label{ueq_pm}
-\Delta_p u\le\frac{1+\epsilon}{u},
\qquad
-\Delta_p u+\vartheta|\nabla u|^p+\vartheta
\ge\frac{1-\epsilon}{u}
\quad\text{in }I_{\delta_0}(\Omega).
\end{equation}
Fix constants $c_1,c_2,L$ in \eqref{asymp=1}, with
$L>1+\max\{0,\ln D_\Omega\}$.

For the upper barrier, choose $\kappa=\kappa_*(1+2\epsilon)^{1/p}$ and,
for a small $\delta<\delta_0$, put
\[
R_\delta=(c_2/\kappa)^p(L-\ln\delta),
\qquad K_\delta=R_\delta+\ln\delta,
\qquad V=h(d_\Omega).
\]
For $0<t\le\delta$, $K_\delta-\ln t\ge R_\delta\to\infty$.
Moreover, if $R_\delta>2$, monotonicity gives
$t(K_\delta-\ln t)^s\le\delta R_\delta^s$ for
$s\in\{1/p,1,(p+1)/p\}$.
Since $R_\delta=O(|\ln\delta|)$, the curvature error in
\eqref{barrier_distance} is bounded by $B\kappa^p\delta R_\delta=o(1)$.
For sufficiently small $\delta$, we can require both
$|\Theta-1|\le\epsilon/4$ and $B\kappa^p\delta R_\delta\le\epsilon/2$,
uniformly on the collar. Thus
\[
(-\Delta_p V)V\ge(1+2\epsilon)(1-\epsilon/4)-\epsilon/2
\ge1+\epsilon.
\]
At $d_\Omega=\delta$ we have
$V=c_2\delta(L-\ln\delta)^{1/p}\ge u$; at $d_\Omega=0$ both vanish.
Lemma \ref{lem:wcp} and \eqref{ueq_pm} yield $u\le V$. With $\delta$
now fixed, $K_\delta$ is a constant, so
\[
\limsup_{d_\Omega(x)\to0}\frac{u(x)}{d_\Omega(x)(-\ln d_\Omega(x))^{1/p}}
\le\kappa_*(1+2\epsilon)^{1/p}.
\]

For the lower barrier, take $\kappa=\kappa_*(1-2\epsilon)^{1/p}$ and
$R_\delta=(c_1/\kappa)^p(L-\ln\delta)$, with $K_\delta$ and $V$ as above.
The sum of the absolute curvature error, the absorption term multiplied by
$V$, and $\vartheta V$ is at most
\[
B\kappa^p\delta R_\delta
+\vartheta\kappa^{p+1}\delta R_\delta^{(p+1)/p}
+\vartheta\kappa\delta R_\delta^{1/p}=o(1).
\]
Reduce $\delta$ so that this is at most $3\epsilon/4$,
$|\Theta-1|\le\epsilon/4$, and, when $\vartheta>0$,
$V\le(1-\epsilon)/\vartheta$ throughout the collar. Then
\[
(-\Delta_p V)V+\vartheta|\nabla V|^pV+\vartheta V
\le(1-2\epsilon)(1+\epsilon/4)+3\epsilon/4\le1-\epsilon.
\]
On the inner boundary $V=c_1\delta(L-\ln\delta)^{1/p}\le u$,
and on $\partial\Omega$ both vanish. Lemma \ref{lem:wcp_singular}
and \eqref{ueq_pm} give $V\le u$, hence
\[
\liminf_{d_\Omega(x)\to0}\frac{u(x)}{d_\Omega(x)(-\ln d_\Omega(x))^{1/p}}
\ge\kappa_*(1-2\epsilon)^{1/p}.
\]
Letting $\epsilon\downarrow0$ proves the uniform limit.
\end{proof}

\section{Directional derivative estimates}\label{sect3}

	In this section, we prove the main results of the paper, namely, Theorems \ref{th:gamma>1}, \ref{th:gamma=1} and \ref{th:gamma<1}.

	\begin{proof}[Proof of Theorem \ref{th:gamma>1}]
		By Theorem \ref{th:asymptotic}, $u\in C^{1,\alpha}_{\rm loc}(\Omega)\cap C(\overline{\Omega})$ for some $0<\alpha<1$.

		(i) First, we prove \eqref{limit>1}. Suppose by contradiction that \eqref{limit>1} does not hold, then there exist $\varepsilon>0$ and subsequences, still denoted by $(x_n)$ and $(\nu_n)$, such that
		\begin{equation}\label{contra_assump>1}
			\left|\delta_n^\frac{\gamma-1}{\gamma+p-1} \frac{\partial u(x_n)}{\partial \nu_n} - \left(\frac{(\gamma + p - 1)^p}{p^{p-1}(p-1)(\gamma - 1)}\right)^{\frac{1}{\gamma + p - 1}} \frac{p\beta}{\gamma + p - 1}\right| \ge \varepsilon,
		\end{equation}
		where we denote
		\[
		\delta_n=d_\Omega(x_n).
		\]

		We start setting up the problem in convenient coordinates as in \cite{MR4044739}.
		Given that the domain is of class \( C^2 \), we can and do restrict our attention to a neighborhood of the boundary where the unique nearest point property is valid. Hence for large $n$, there is a unique point $\hat x_n$ such that $|x_n-\hat x_n|=d_\Omega(x_n)$.
		For each $n$, we choose a new coordinate such that in this coordinate, $\hat x_n=0$ and \( \eta(x_n) = e_N \). More precisely, for each \( n \in \mathbb{N} \), we can define an isometry \( T_n: \mathbb{R}^N \to \mathbb{R}^N \) that possesses the aforementioned properties by simply combining a translation and a rotation of the coordinate axes. Write $T_n(x)=Q_n(x-\hat x_n)$ with $Q_n$ orthogonal and $Q_n\eta(x_n)=e_N$. This process generates a new sequence of points \(y_n := T_n (x_n) = \delta_n e_N\). Defining \( u_n(y) := u(T_n^{-1}(y)) \), it follows that
		\[
		-\Delta_p u_n + \vartheta |\nabla u_n|^q = \frac{1}{u_n^\gamma} + f(u_n) \quad\text{ in } \Omega_n,
		\]
		where $\Omega_n=T_n(\Omega)$. By Theorem \ref{th:asymptotic}, there exist $c_1,c_2>0$ such that for any $n$,
		\begin{equation}\label{un_bounds>1}
			c_1 d_{\Omega_n}(x)^\frac{p}{\gamma+p-1} \le u_n(x) \le c_2 d_{\Omega_n}(x)^\frac{p}{\gamma+p-1} \quad\text{ in } \Omega_n.
		\end{equation}

		Now we set
		\[
		w_n(y) := \delta_n^{-\frac{p}{\gamma+p-1}} u_n(\delta_n y) \quad\text{ for } y \in \tilde{\Omega}_n,
		\]
		where
		\[
		\tilde{\Omega}_n := \Omega_n/\delta_n.
		\]
		Then $w_n$ weakly solves
		\begin{equation}\label{wn_eq>1}
			-\Delta_p w_n + \delta_n^{\frac{\gamma (p - q) + q}{\gamma + p - 1}} \vartheta |\nabla w_n|^q = \frac{1}{w_n^\gamma} + \delta_n^{\frac{p \gamma}{\gamma + p - 1}} f\left( \delta_n^{\frac{p}{\gamma + p - 1}} w_n \right) \quad \text{ in } \tilde{\Omega}_n.
		\end{equation}

		We claim that up to a subsequence
		\begin{equation}\label{holder_convergence>1}
			w_n \to w_\infty \text{ in } C^{1,\alpha'}_{\rm loc}(\mathbb{R}^N_+) \quad\text{ as } n\to\infty \text{ for some } 0<\alpha'<1.
		\end{equation}
		To prove this, let any compact set $K \subset \mathbb{R}^N_+$. For $n$ large we have $K\subset \tilde\Omega_n$, or equivalently, $\delta_n K\subset \Omega_n$. Hence for any $y\in K$, we can apply \eqref{un_bounds>1} to get
		\[
		c_1 \delta_n^{-\frac{p}{\gamma+p-1}} d_{\Omega_n}(\delta_n y)^\frac{p}{\gamma+p-1} \le w_n(y) \le c_2 \delta_n^{-\frac{p}{\gamma+p-1}} d_{\Omega_n}(\delta_n y)^\frac{p}{\gamma+p-1}.
		\]

		On the other hand, we have
		\[
		d_{\Omega_n}(\delta_n y) = \delta_n d_{\tilde\Omega_n}(y)
		\]
		and
		\[
		d_{\tilde\Omega_n}(y) \to d_{\mathbb{R}^N_+}(y) = y_N.
		\]
		Here we use that $\tilde\Omega_n$ converges to $\mathbb{R}^N_+$ locally in $C^2$, since $\partial\Omega$ is of class $C^2$ and the curvature of $\partial\tilde\Omega_n$ is $O(\delta_n)$. Consequently, for large $n$,
		\begin{equation}\label{wn_bounds>1}
			\frac{c_1}{2} y_N^\frac{p}{\gamma+p-1} \le w_n(y) \le 2c_2 y_N^\frac{p}{\gamma+p-1} \quad \text{ for all } y\in K.
		\end{equation}
		Hence for any compact set $K \subset \mathbb{R}^N_+$, we have $C_1(K) \le w_n(y) \le C_2(K)$ for all $y\in K$, where $C_1(K),C_2(K)$ are two positive constants depending on $K$. The displayed right-hand side of \eqref{wn_eq>1} is uniformly bounded on a slightly larger compact set. The gradient coefficient is uniformly bounded and tends to zero, since $\gamma(p-q)+q>0$; moreover $|\xi|^q\le1+|\xi|^p$. Thus the structural constants for the interior natural-growth estimates are uniform. By \cite{MR709038,MR727034}, it follows that $(w_n)$ is uniformly bounded in $C^{1,\alpha}(K)$. By the Ascoli--Arzel\`a theorem and a diagonal argument, after passing to a subsequence we derive $w_n \to w_\infty \text{ in } C^{1,\alpha'}_{\rm loc}(\mathbb{R}^N_+)$ as $n\to\infty \text{ for } 0<\alpha'<\alpha$. Hence \eqref{holder_convergence>1} is proved.

		Letting $n\to\infty$ in \eqref{wn_bounds>1}, we deduce
		\[
		\frac{c_1}{2} y_N^\frac{p}{\gamma+p-1} \le w_\infty(y) \le 2c_2 y_N^\frac{p}{\gamma+p-1} \quad\text{ for } y\in\mathbb{R}^N_+.
		\]
		In particular, we can extend $w_\infty$ by zero on $\partial\mathbb{R}^N_+$, then $w_\infty \in C^{1,\alpha'}_{\rm loc}(\mathbb{R}^N_+) \cap C(\overline{\mathbb{R}^N_+})$. Letting $n\to\infty$ in \eqref{wn_eq>1}, we have that $w_\infty$ weakly solves
		\[
		\begin{cases}
			-\Delta_p w_\infty = \frac{1}{w_\infty^\gamma} &\text{ in } \mathbb{R}^N_+,\\
			w_\infty>0 &\text{ in } \mathbb{R}^N_+,\\
			w_\infty=0 &\text{ on } \partial\mathbb{R}^N_+.
		\end{cases}
		\]
		By \cite[Theorem 1.2]{MR4044739}, whose only extra hypothesis is a two-sided bound of the form just displayed, we have
		\[
		w_\infty(x) \equiv \left(\frac{(\gamma+p-1)^p}{p^{p-1}(p-1)(\gamma-1)}\right)^\frac{1}{\gamma+p-1}	x_N^\frac{p}{\gamma+p-1}.
		\]

		On the other hand, setting $\tilde\nu_n=Q_n\nu_n\in \mathbb{S}^{N-1}$, we have
		\[
		\langle \tilde\nu_n, e_N \rangle = \langle \nu_n, \eta(x_n) \rangle \to \beta.
		\]
		Up to a subsequence, we may assume $\tilde\nu_n \to \tilde{\nu_0} \in \mathbb{S}^{N-1}$ with $\langle \tilde{\nu_0}, e_N \rangle = \beta$.
		Therefore,
		\begin{align*}
			\delta_n^\frac{\gamma-1}{\gamma+p-1} \frac{\partial u(x_n)}{\partial \nu_n} &= \delta_n^\frac{\gamma-1}{\gamma+p-1} \frac{\partial u_n(y_n)}{\partial \tilde\nu_n} =
			\frac{\partial w_n(e_N)}{\partial \tilde\nu_n}\\
			& \to \frac{\partial w_\infty(e_N)}{\partial \tilde{\nu_0}}=\left(\frac{(\gamma + p - 1)^p}{p^{p-1}(p-1)(\gamma - 1)}\right)^{\frac{1}{\gamma + p - 1}} \frac{p\beta}{\gamma + p - 1}.
		\end{align*}
		However, this contradicts \eqref{contra_assump>1}.
		Hence \eqref{limit>1} is proved. The same compactness argument, evaluated
at $e_N$, gives $w_n(e_N)\to A_\gamma$ and
$\nabla w_n(e_N)\to a_\gamma A_\gamma e_N$ for every boundary sequence.
Since $Q_n$ is orthogonal, returning to the original coordinates proves both
limits in \eqref{vector>1}, uniformly as $d_\Omega\to0$.

		(ii) The uniform gradient limit in \eqref{vector>1} immediately gives
\eqref{grad_bound>1} in a boundary collar. On the remaining interior compact
set the gradient is bounded by interior regularity. Increasing $C$ proves
\eqref{grad_bound>1} throughout $\Omega$.

		Lastly, from \eqref{grad_bound>1} and Theorem \ref{th:asymptotic}, we deduce
		\[
		\left|\nabla\left(u^\frac{\gamma+p-1}{p}\right)\right| = \frac{\gamma+p-1}{p} u^\frac{\gamma-1}{p} |\nabla u| \le C \frac{\gamma+p-1}{p} c_2^\frac{\gamma-1}{p}.
		\]
		Hence $u^\frac{\gamma+p-1}{p}$ is Lipschitz continuous on $\overline\Omega$, because a bounded $C^2$ domain is quasiconvex. Since $|s^{a_\gamma}-t^{a_\gamma}|\le|s-t|^{a_\gamma}$ for $s,t\ge0$, we conclude $u \in C^{\frac{p}{\gamma+p-1}}(\overline{\Omega})$.
	\end{proof}

	\begin{proof}[Proof of Theorem \ref{th:gamma=1}]
Let $x_n\in\Omega$ be any sequence with $\delta_n=d_\Omega(x_n)\to0$.
Use the isometries $T_n(x)=Q_n(x-\hat x_n)$ from the preceding proof and set
\[
u_n=u\circ T_n^{-1},\qquad \Omega_n=T_n(\Omega),\qquad
\ell_n=L-\ln\delta_n,\qquad
w_n(y)=\frac{u_n(\delta_n y)}{\delta_n\ell_n^{1/p}}.
\]
On $\widetilde\Omega_n=\Omega_n/\delta_n$ the equation is
\begin{equation}\label{wn_eq=1}
\begin{aligned}
-\Delta_p w_n+\vartheta\delta_n\ell_n^{(q-p+1)/p}|\nabla w_n|^q
&=\frac{\ell_n^{-1}}{w_n}
+\delta_n\ell_n^{-(p-1)/p}f(\delta_n\ell_n^{1/p}w_n).
\end{aligned}
\end{equation}
The two-sided bounds in Theorem \ref{th:asymptotic}(ii), together with
$d_{\Omega_n}(\delta_n y)=\delta_n d_{\widetilde\Omega_n}(y)$ and
$d_{\widetilde\Omega_n}\to y_N$ locally uniformly in $\mathbb R^N_+$,
bound $w_n$ above and away from zero on every compact subset of that half-space.
The displayed right-hand side is uniformly bounded there, and the coefficient
of the natural-growth term tends to zero. The same interior estimates as above
give, after passage to a subsequence,
\[
w_n\longrightarrow w_\infty
\quad\text{in }C^{1,\alpha'}_{\rm loc}(\mathbb R^N_+)
\quad\text{for some }\alpha'>0.
\]
Theorem \ref{th:sharp=1} identifies this limit directly. For each fixed
$y\in\mathbb R^N_+$,
\[
\begin{aligned}
w_n(y)
&=\frac{u_n(\delta_n y)}{d_{\Omega_n}(\delta_n y)
(-\ln d_{\Omega_n}(\delta_n y))^{1/p}}
\frac{d_{\Omega_n}(\delta_n y)}{\delta_n}
\left(\frac{-\ln d_{\Omega_n}(\delta_n y)}{\ell_n}\right)^{1/p}\\
&\longrightarrow\kappa_* y_N.
\end{aligned}
\]
Consequently $\nabla w_n(e_N)\to\kappa_*e_N$.
The linear form of the limit also agrees with the half-space classification
in \cite[Theorem 3.1]{MR2337497}; the value asymptotics above identify its coefficient.
Returning to the original coordinates and using $\ell_n/(-\ln\delta_n)\to1$
yields \eqref{vector=1}. Since the sequence was arbitrary, the convergence
is uniform. Taking scalar products with $\nu_n$ proves \eqref{limit=1}.
The cone estimate \eqref{bound=1}, with the constants stated in (iii), follows
by taking scalar products with directions satisfying $\nu\cdot\eta\ge\beta$.
The normalized gradient is uniformly bounded in a boundary collar; interior
regularity then gives \eqref{grad_bound=1} throughout $\Omega$ after increasing
$L$ and $C$.

Finally, for $0<s<1$, the bounds for $u$ and $\nabla u$ imply
\[
\left|\nabla(u^{1/s})\right|
=\frac1s u^{1/s-1}|\nabla u|
\le C d_\Omega^{1/s-1}(L-\ln d_\Omega)^{1/(ps)}\le C_s.
\]
Thus $u^{1/s}$ is Lipschitz continuous on the bounded $C^2$ domain, and
$u\in C^s(\overline\Omega)$.
\end{proof}

\begin{remark}\label{rem:loglip}
		In the borderline case $\gamma=1$, the conclusion $u\in C^{s}(\overline{\Omega})$ for every $s\in(0,1)$ can be upgraded to the following sharp \emph{log-Lipschitz} estimate: there is $C>0$ such that
		\begin{equation}\label{loglip}
			|u(x)-u(y)| \le C\,|x-y|\left(1+\left|\ln|x-y|\right|\right)^\frac{1}{p} \quad \text{ for all } x,y \in \overline{\Omega} \text{ with } 0<|x-y|<\tfrac{1}{2}.
		\end{equation}
		Indeed, put $\varrho=|x-y|$ and first take $\varrho$ smaller than a fixed
tubular radius. If $\min\{d_\Omega(x),d_\Omega(y)\}\le2\varrho$, the
one-Lipschitz property of $d_\Omega$ gives
$d_\Omega(x),d_\Omega(y)\le3\varrho$. The value bound
\eqref{asymp=1} and positivity therefore imply
\[
|u(x)-u(y)|\le u(x)+u(y)
\le C\varrho(1+|\ln\varrho|)^{1/p}.
\]
Otherwise $d_\Omega(x),d_\Omega(y)>2\varrho$. The segment $[x,y]$
lies in $B_\varrho(x)\subset\Omega$, and its distance from the boundary
is at least $\varrho$. Integrating \eqref{grad_bound=1} on this segment
gives the same bound. For the remaining values of $\varrho<1/2$,
increase $C$ using boundedness of $u$.
Estimate \eqref{loglip} is optimal: for $x_0\in\partial\Omega$ and
$x=x_0+t\eta(x_0)$, Theorem \ref{th:sharp=1} gives
$|u(x)-u(x_0)|\sim\kappa_*t(-\ln t)^{1/p}$ as $t\downarrow0$.
Thus no modulus $o(t(-\ln t)^{1/p})$ is possible. In particular, $u$ is
\emph{not} Lipschitz continuous on $\overline\Omega$, whereas
\eqref{loglip} implies $u\in C^s(\overline{\Omega})$ for every $s\in(0,1)$.
	\end{remark}

	\begin{proof}[Proof of Theorem \ref{th:gamma<1}]
By Theorem \ref{th:asymptotic}(iii),
\begin{equation}\label{u_bounds<1}
c_1d_\Omega(x)\le u(x)\le c_2d_\Omega(x).
\end{equation}
We first bound the full gradient, before invoking boundary regularity.
For $x\in\Omega$ near the boundary put $r=d_\Omega(x)/2$ and
$v(y)=r^{-1}u(x+ry)$ for $y\in B_1$. On this ball,
$r\le d_\Omega(x+ry)\le3r$, so $c_1\le v\le3c_2$, and
\[
-\Delta_p v+\vartheta r|\nabla v|^q
=r^{1-\gamma}v^{-\gamma}+r f(rv)\quad\text{in }B_1.
\]
Because $\gamma<1$, the displayed right-hand side is bounded uniformly as
$r\downarrow0$; the natural-growth coefficient $\vartheta r$ is also bounded.
The interior estimates of \cite{MR709038,MR727034} give $|\nabla v(0)|\le C$.
Since $\nabla v(0)=\nabla u(x)$, compactness in the interior yields
\begin{equation}\label{grad_bound<1}
\|\nabla u\|_{L^\infty(\Omega)}<\infty.
\end{equation}
In particular, $u\in W^{1,p}_0(\Omega)$, and now the full right-hand side
\[
H(x):=u(x)^{-\gamma}+f(u(x))-\vartheta|\nabla u(x)|^q
\]
satisfies $|H(x)|\le C d_\Omega(x)^{-\gamma}$.
Moreover $H\ge0$ in a sufficiently thin boundary collar, since the singular
term tends uniformly to $+\infty$ and the other two terms are bounded.
The boundary estimate of \cite[Appendix B, Theorem B.1]{MR2341518} is stated
for $W^{1,p}_0$ solutions with $0\le u\le Cd_\Omega$ and
$0\le H\le Cd_\Omega^{-\gamma}$, $\gamma<1$, in all of $\Omega$. Its proof flattens
$\partial\Omega$ and works in half-balls, so it applies on a boundary collar where
$H\ge0$, and the sign of $H$ plays no role in the remaining interior compact set.
Combining it with interior regularity proves $u\in C^{1,\alpha}(\overline\Omega)$.

For $x_0\in\partial\Omega$, divide the lower bound in \eqref{u_bounds<1}
at $x_0+t\eta(x_0)$ by $t$ and let $t\downarrow0$. It follows that
$\partial_\eta u(x_0)\ge c_1$. The tangential derivatives vanish because
the boundary trace is constant. Hence
$\partial_\nu u(x_0)=(\nu\cdot\eta(x_0))\partial_\eta u(x_0)\ge\beta c_1$,
which proves \eqref{hopf<1}.
\end{proof}

\begin{proof}[Proof of Corollary \ref{co:W1p0}]
For $\gamma>1$, \eqref{vector>1} gives
$|\nabla u(x)|\asymp d_\Omega(x)^{-(\gamma-1)/(\gamma+p-1)}$
uniformly near the boundary. Integration in tubular coordinates proves the stated
criterion, including divergence at equality. For $\gamma=1$ use
\eqref{grad_bound=1} and the integrability of every finite power of $|\ln t|$
on $(0,\delta)$. For $0<\gamma<1$ use Theorem \ref{th:gamma<1}.
In each finite-integral case, continuity up to the boundary and the zero boundary
value imply the zero Sobolev trace.
\end{proof}

\section{Symmetry of solutions}\label{sect4}
We complete the moving-plane argument using estimates localized in the
interior of $\Omega$. We record the estimates and the comparison argument
explicitly, including the dependence on $q$. In particular, the argument
does not require a strong comparison principle across the critical set.
Throughout this section write $F(s)=s^{-\gamma}+f(s)$ and
$Z_u=\{x\in\Omega:\nabla u(x)=0\}$.

\subsection{Interior estimates and comparison}
The following local estimates are the ones used in the qualitative theory
of quasilinear equations; see \cite{MR4022661,EMM2021}. We include the
exponent calculation that permits the range $q\ge p/2$ and explain how the
positivity assumption enters.

\begin{lemma}[Localized weighted estimates]\label{lem:local-weight}
Let $N\ge2$, $V\subset\mathbb R^N$ be open, and let
$v\in C^{1,\alpha}_{\rm loc}(V)$ be a positive solution of
\[
-\Delta_pv+\vartheta|\nabla v|^q=F(v)\quad\text{in }V.
\]
Set $Z_v=\{x\in V:\nabla v(x)=0\}$.
Assume that $F$ is locally Lipschitz on the range of $v$, and, when
$\vartheta>0$, that $q\ge\max\{1,p/2\}$. Fix $U\subset\subset V$.
For $0\le b<1$, one has
\begin{equation}\label{weighted-hessian}
\sup_{y\in U}\int_{U\setminus Z_v}
\frac{|\nabla v|^{p-2-b}|D^2v|^2}{|x-y|^\sigma}\,dx<\infty,
\end{equation}
where $0\le\sigma<N-2$ if $N\ge3$, and $\sigma=0$ if $N=2$.
If $F(v)\ge m>0$ on a neighborhood of $\overline U$, then, for every
$0<\tau<p-1$,
\begin{equation}\label{inverse-gradient}
\sup_{y\in U}\int_U
\frac{dx}{|\nabla v|^\tau|x-y|^\sigma}<\infty.
\end{equation}
Consequently $|Z_v\cap U|=0$, and
$|\nabla v|^{p-2}|D^2v|\in L^1(U\setminus Z_v)$.
\end{lemma}
\begin{proof}
Choose $U\subset\subset V_0\subset\subset V$ and a smooth cutoff $\chi$
supported in $V_0$ and equal to one on $U$.
All constants below depend only on these sets, $p,q,\vartheta$, the
$C^1$ bound for $v$ on $V_0$, and the Lipschitz bound for $F$ on its range;
the inverse estimate also depends on the positive lower bound $m$.
Away from $Z_v$ the equation is uniformly elliptic locally. The weighted
Hessian estimate in \cite[Theorem 2.1]{EMM2021} is stated for
$q\ge\max\{p-1,1\}$ and a $C^1$ source. For the range used here,
the standard regularized difference-quotient argument tests the
differentiated equation with
\[
\frac{v_{x_i}}{(\epsilon+|v_{x_i}|)^b}\,\zeta^2,
\qquad
\zeta(x)=\chi(x)(|x-y|^2+\delta^2)^{-\sigma/4},
\]
and then lets $\epsilon,\delta\downarrow0$.
We give the term that controls the admissible gradient exponent.
When $\vartheta=0$ all absorption terms below are omitted; otherwise
$2q-p\ge0$ by assumption.
With $s=|\nabla v|$, differentiation of the absorption contributes at most
\[
\vartheta q\int s^{q-1}|\nabla v_{x_i}|
\frac{|v_{x_i}|}{(\epsilon+|v_{x_i}|)^b}\zeta^2.
\]
Young's inequality bounds this by an arbitrarily small multiple of
\[
\int s^{p-2}\frac{|\nabla v_{x_i}|^2}
{(\epsilon+|v_{x_i}|)^b}\zeta^2
\]
plus a constant times
\[
\int s^{2q-p}
\frac{|v_{x_i}|^2}{(\epsilon+|v_{x_i}|)^b}\zeta^2.
\]
The latter is uniformly bounded because $2q-p\ge0$ and $b<1$.
In this weighted Hessian calculation, the absorption term therefore
requires only $2q\ge p$, besides the standing assumption $q\ge1$.
For $p>2$, the lower bound $q\ge p/2$ is weaker than $q\ge p-1$.
The differentiated source is bounded using $\operatorname{Lip}F$.
The cutoff errors are bounded by a constant times
$\int s^{p-b}|\nabla\zeta|^2$, which is finite uniformly in $y,\delta$:
the only singular spatial power is $|x-y|^{-\sigma-2}$ and
$\sigma<N-2$. For $N=2$ there is no singular kernel.
Ellipticity, summation over $i$, and
$(\epsilon+|v_{x_i}|)^{-b}\ge(\epsilon+s)^{-b}$ now yield
\eqref{weighted-hessian}. These estimates are obtained with the usual
regularizations at $s=0$; they concern the classical Hessian off $Z_v$.
Smooth approximation of $F$ uses only its local Lipschitz bound, so the
argument does not require $F\in C^1$.

For completeness, the inverse estimate follows by testing the original
equation with the regularization of
$\chi^2 s_\epsilon^{-\tau}(|x-y|^2+\delta^2)^{-\sigma/2}$, where
$s_\epsilon=(\epsilon^2+s^2)^{1/2}$. Here $V_0$ and $\chi$ are chosen
inside the neighborhood where $F(v)\ge m$. To justify this test, one can first
truncate $s$ below a positive level and then let that level decrease to zero.
Choose
\[
\max\{0,2-p+\tau\}<b<1;
\]
this is possible precisely because $\tau<p-1$.
The term containing $D^2v$ is bounded by Cauchy--Schwarz and
\eqref{weighted-hessian}. Indeed, its remaining squared factor is bounded
by $s_\epsilon^{p-2+b-2\tau}$, and
$p-2+b-2\tau>-\tau$. The derivative of the spatial cutoff is harmless
because $p-1-\tau>0$. Finally, for any $a>0$,
$s^q s_\epsilon^{-\tau}\le a s_\epsilon^{-\tau}+C_a$.
Writing $X_\epsilon$ for the integral to be estimated, with the cutoffs
and regularized kernel, gives
\[
mX_\epsilon\le C+C(1+X_\epsilon)^{1/2}
+\vartheta aX_\epsilon.
\]
Take $\vartheta a<m/2$ when $\vartheta>0$. This bounds
$X_\epsilon$ uniformly; Fatou's lemma proves \eqref{inverse-gradient}
and $|Z_v\cap U|=0$.
Finally choose $b\in(\max\{0,2-p\},1)$ in
\eqref{weighted-hessian}, with $\sigma=0$. Cauchy--Schwarz gives
\[
\int_{U\setminus Z_v}s^{p-2}|D^2v|
\le\left(\int_{U\setminus Z_v}s^{p-2-b}|D^2v|^2\right)^{1/2}
\left(\int_U s^{p-2+b}\right)^{1/2}<\infty.
\]
\end{proof}

\begin{lemma}[Comparison on sets of small measure]\label{lem:small-compare}
Assume $N\ge2$ and, when $\vartheta>0$, $q\ge\max\{1,p/2\}$.
Let $U\subset\subset V\subset\subset\Omega$. Suppose that $v$ is another solution
on a neighborhood of $\overline U$, and that $u,v$ and their gradients
have fixed bounds there, with $u,v$ bounded away from zero.
Under \eqref{positive-source}, there exists $\mu>0$ such that
\[
u\le v\ \text{on }\partial D,\qquad |D|<\mu
\quad\Longrightarrow\quad u\le v\ \text{in }D
\]
for open $D\subset U$ for which $(u-v)^+\in W^{1,p}_0(D)$.
The threshold is uniform for a family of $v$ with these fixed bounds.
\end{lemma}
\begin{proof}
We first record a Poincar\'e estimate for the fixed solution $u$.
When $p>2$, put $\rho=|\nabla u|^{p-2}$ and choose
\[
1<t<\frac{p-1}{p-2}.
\]
Lemma \ref{lem:local-weight} gives
$\sup_{x\in U}\int_U\rho(y)^{-t}|x-y|^{-\sigma}\,dy<\infty$.
Choose $\sigma$ sufficiently close to $N-2$ that
$2t-N+\sigma>0$; for $N=2$ take $\sigma=0$.
The representation estimate for $\phi\in C^1_c(D)$ and H\"older's
inequality with exponents $2t$ and $r=2t/(2t-1)$ give
\[
|\phi(x)|\le C
\left(\int_D
\frac{(\rho^{1/2}|\nabla\phi|)^r}{|x-y|^{N-\alpha}}\,dy
\right)^{1/r},
\qquad \alpha=\frac{2t-N+\sigma}{2t-1}>0.
\]
The Riesz-potential estimate implies
\[
\|\phi\|_{L^Q(U)}\le C\left(\int_D\rho|\nabla\phi|^2\right)^{1/2},
\qquad \frac1Q=\frac12-\frac{\alpha}{Nr}<\frac12.
\]
Consequently, by density,
\begin{equation}\label{small-poincare}
\int_D\phi^2\le C|D|^{1-2/Q}\int_D\rho|\nabla\phi|^2.
\end{equation}
All constants refer to the fixed sets and fixed solution $u$.
For $1<p\le2$, ordinary Poincar\'e--Sobolev gives the same conclusion
with $\rho=1$ and some exponent $Q>2$.

Now test the two equations with $w=(u-v)^+$ and subtract.
Set $S=|\nabla u|+|\nabla v|$ and
$E=\int_D S^{p-2}|\nabla w|^2$, interpreting the integrand as zero
where $S=0$. On the compact solution range, $F$ is Lipschitz, so
\begin{align*}
\Lambda E
&\le L\int_Dw^2
+\vartheta q\int_D S^{q-1}|\nabla w|w\\
&\le\frac\Lambda2 E+
C\int_D(1+S^{2q-p})w^2
\le\frac\Lambda2E+C\int_Dw^2.
\end{align*}
Here the last inequality uses exactly $2q-p\ge0$ and the gradient bounds.
For $p>2$, $S^{p-2}\ge\rho$. For $1<p\le2$, $S^{p-2}$ is bounded
below by a positive constant wherever $S>0$.
Thus \eqref{small-poincare} gives $E\le C|D|^{1-2/Q}E$.
Choosing $|D|$ small yields $E=0$, and hence $w=0$.
When $\vartheta=0$, the terms involving $q$ are absent.
\end{proof}

\begin{lemma}[Exclusion of interior regions bounded by critical points]
\label{lem:no-interior-region}
Under the assumptions of Theorem \ref{th:symmetry}, with $N\ge2$, there is no nonempty
open set $G\subset\subset\Omega$ with $\partial G\subset Z_u$.
\end{lemma}
\begin{proof}
Write $s=|\nabla u|$ and let $m=\min_{\overline G}F(u)>0$.
Choose $A\ge0$ so that
$\vartheta t^q\le At^p+m/2$ for $0\le t\le\|\nabla u\|_{L^\infty(G)}$.
This follows from continuity at $t=0$ and boundedness of the interval.
Let $h_\epsilon$ be zero on $[0,\epsilon]$, equal to one on
$[2\epsilon,\infty)$, and linear in between. The function
$\psi_\epsilon=e^{-Au}h_\epsilon(s)\mathbf1_G$ is an admissible
compactly supported test: $s=0$ on $\partial G$, and its support stays
away from that boundary and from $Z_u$ for fixed $\epsilon$.
Testing the equation gives
\[
\int_G e^{-Au}\bigl(F(u)+As^p-\vartheta s^q\bigr)h_\epsilon(s)
=\int_G e^{-Au}h_\epsilon'(s)|\nabla u|^{p-2}\nabla u\cdot\nabla s.
\]
The absolute value of the right-hand side is at most
\[
2\int_{G\cap\{\epsilon<s<2\epsilon\}}s^{p-2}|D^2u|\longrightarrow0
\]
by Lemma \ref{lem:local-weight} and dominated convergence.
Since $|Z_u\cap G|=0$, the left-hand side tends to a quantity at least
\[
\frac m2\int_G e^{-Au}\,dx>0.
\]
This contradiction proves the assertion. The test-function mechanism is
also used in \cite[Section 3]{EMM2021}; here it is entirely confined to
an interior compact set and requires only positivity of the total source.
\end{proof}

\subsection{The moving-plane argument}
For $a=\inf_\Omega x_1$ and $a<\lambda\le0$, set
$T_\lambda=\{x\in\mathbb R^N:x_1=\lambda\}$ and
\[
\Omega_\lambda=\Omega\cap\{x_1<\lambda\},\qquad
R_\lambda(x)=(2\lambda-x_1,x_2,\ldots,x_N),\qquad
u_\lambda=u\circ R_\lambda.
\]
Directional convexity and symmetry imply $R_\lambda(\Omega_\lambda)
\subset\Omega$.

\begin{proof}[Proof of Theorem \ref{th:symmetry}]
Assume first $N\ge2$.
The boundary estimates already proved imply
$\partial_{e_1}u>0$ near every compact subset of
$\partial\Omega\cap\{x_1<0\}$ where the inward normal has positive
$e_1$ component. They also imply that $Z_u$ is contained in an
interior compact set. Lemma \ref{lem:local-weight}, applied on that set,
gives $|Z_u|=0$.

We spell out a boundary observation used below. If
$J\subset\subset(a,0)$ is a compact interval, then there exists $\delta_J>0$
such that, uniformly for $\lambda\in J$,
\begin{equation}\label{s3}
u<u_\lambda\quad\text{in }I_{\delta_J}(\Omega)\cap\Omega_\lambda.
\end{equation}
Indeed, near $\partial\Omega\cap T_\lambda$, strict directional
convexity gives a uniform positive lower bound for $\eta\cdot e_1$.
For a sufficiently short horizontal segment crossing $T_\lambda$,
the whole segment lies in the collar where $\partial_{e_1}u>0$.
Integration along the segment proves the comparison there.
For the remaining boundary points, which stay a fixed horizontal
distance from $T_\lambda$, their reflected points range over a compact
subset of $\Omega$. The positive minimum of $u$ on that set and the
uniform boundary limit $u\to0$ give the comparison after reducing
$\delta_J$. Compactness in $\lambda\in J$ makes these choices uniform.

When $\lambda>a$ is sufficiently close to $a$, both $\Omega_\lambda$
and its reflection lie near the leftmost boundary points, whose inward
normal is $e_1$. Hence $\partial_{e_1}u>0$ along the horizontal segments
joining each $x$ to $R_\lambda x$, and $u<u_\lambda$ in
$\Omega_\lambda$. Define
\[
\lambda_0=\sup\{\lambda\in(a,0):
u\le u_t\text{ in }\Omega_t\text{ for every }a<t\le\lambda\}.
\]
Suppose $\lambda_0<0$. By continuity $u\le u_{\lambda_0}$ in
$\Omega_{\lambda_0}$.

On a connected component $C$ of $\Omega_{\lambda_0}\setminus Z_u$,
the ordinary local strong comparison principle (cf.\ \cite{MR2356201}) gives either
$u<u_{\lambda_0}$ or $u\equiv u_{\lambda_0}$.
This local principle is used only at contact points with nonzero gradient:
there the gradients of the two functions agree and are nonzero, so
linearization is uniformly elliptic in a neighborhood. It holds for every
$p>1$, with no dimensional restriction on $p$.
The equality alternative is impossible. By \eqref{s3}, the closure of
such a component cannot meet the physical boundary.
Reflect $C$ across $T_{\lambda_0}$ and adjoin the noncritical points
of that plane in their common closure. The resulting set $G$ is open:
near any such plane point the noncritical neighborhood has its left
half in $C$ and its right half in $R_{\lambda_0}C$.
Moreover $G\subset\subset\Omega$ and $\partial G\subset Z_u$.
For the reflected part this follows from the identity
$u=u_{\lambda_0}$ and the corresponding equality of gradient norms.
Lemma \ref{lem:no-interior-region} excludes $G$. Consequently
\begin{equation}\label{s1}
u<u_{\lambda_0}\quad\text{in }\Omega_{\lambda_0}\setminus Z_u.
\end{equation}

Choose $J\subset\subset(a,0)$ containing $\lambda_0$ in its interior.
By \eqref{s3}, every positive part
$w_\lambda=(u-u_\lambda)^+$ for $\lambda$ near $\lambda_0$ vanishes
in a fixed collar of the physical boundary.
Its trace on $T_\lambda$ is zero, so
$w_\lambda\in W^{1,p}_0(\Omega_\lambda)$.
The points where it can be positive and their reflected points lie in
fixed interior compact sets. Thus both solutions are bounded away from
zero there and their gradients have fixed bounds.
Choose the small-measure threshold $\mu$ in
Lemma \ref{lem:small-compare} using these bounds and the fixed solution $u$.
Since $|Z_u|=0$, choose a compact
$K\subset\Omega_{\lambda_0}\setminus Z_u$ with
$|\Omega_{\lambda_0}\setminus K|<\mu/2$.
By \eqref{s1} and uniform continuity, for all sufficiently small $h>0$,
\[
u<u_{\lambda_0+h}\quad\text{on }K,
\qquad |\Omega_{\lambda_0+h}\setminus K|<\mu.
\]
Apply Lemma \ref{lem:small-compare} to the positive part outside $K$
and outside the collar where it already vanishes. Its trace on the
remaining boundary, including the reflecting plane, is zero.
We obtain $w_{\lambda_0+h}=0$, contradicting the definition of
$\lambda_0$. Therefore $\lambda_0=0$.
Moving from the opposite side proves symmetry; varying the reflecting
plane gives $\partial_{x_1}u\ge0$ in $\Omega\cap\{x_1<0\}$.

When $N=1$, write $\Omega=(-b,b)$ and $M=\max u$.
The flux $J=|u'|^{p-2}u'$ is $C^1$ in the interior and satisfies
$J'=\vartheta|u'|^q-F(u)$. At every zero of $J$, the derivative is
strictly negative. Hence there is exactly one critical point $c$,
with $u'>0$ to its left and $u'<0$ to its right.
On each branch put $s=u(x)$ and $V(s)=|u'(x)|^p$. Then
\[
\frac{p-1}{p}V'(s)=\vartheta V(s)^{q/p}-F(s),\qquad V(M)=0.
\]
After the change $t=M-s$, the right-hand side is nonincreasing in $V$.
The derivative of the squared difference of two solutions with the same
initial value is therefore nonpositive. Both branches have the same $V$
and the same length $\int_0^M V(s)^{-1/p}\,ds$. Thus $c=0$ and $u$ is even.

Finally, in a ball, symmetry in every direction gives radial symmetry.
Writing $u(x)=U(r)$, monotonicity gives $U'\le0$ for $r>0$.
If $U'(r_0)=0$ for some $r_0>0$, the radial flux
$J(r)=r^{N-1}|U'|^{p-2}U'$ satisfies
$J'(r_0)=-r_0^{N-1}F(U(r_0))<0$.
It would be positive immediately to the left of $r_0$, a contradiction.
Hence $U'(r)<0$ for $r>0$.
\end{proof}

\section*{Acknowledgments}
	This research was conducted during the author's visit to the Vietnam Institute for Advanced Study in Mathematics (VIASM) in 2025. The author gratefully acknowledges the institute's hospitality and generous support during this period.

\section*{Statements and Declarations}

\noindent\textbf{Competing interests.} The author declares no conflict of interest.

\smallskip
\noindent\textbf{Data availability.} No data were generated or analyzed in this study.

\end{document}